\documentclass[12pt,a4paper]{article}
\usepackage[english]{babel}
\usepackage{graphicx}
\usepackage[svgnames]{xcolor}
\usepackage{indentfirst}
\usepackage[colorlinks=true,linkcolor=blue,citecolor=blue]{hyperref}
\usepackage{caption}
\usepackage{listings} 
\usepackage{amsmath}
\usepackage{amssymb}
\usepackage{amsthm} 
\usepackage{graphicx,geometry}
\usepackage{enumitem}
\usepackage{subfig}
\usepackage{mathtools}
\usepackage{float}
\usepackage{lpic}
\usepackage{tikz}
\usepackage{dsfont}
\usetikzlibrary{math,patterns}
\makeatletter

\def\env@sqcases{%
  \let\@ifnextchar\new@ifnextchar
  \left\lbrack
  \def\arraystretch{1.2}%
  \array{@{}l@{\quad}l@{}}%
}
\makeatother

\usepackage{titlesec}

\titleformat{\paragraph}
{\normalfont\normalsize\bfseries}{\theparagraph}{1em}{}
\titlespacing*{\paragraph}
{0pt}{3.25ex plus 1ex minus .2ex}{1.5ex plus .2ex}

\newtheorem{theorem}{Theorem}
\newtheorem{lemma}{Lemma}
\newtheorem{proposition}{Proposition}

\theoremstyle{definition}
\newtheorem{definition}{Definition}
\newtheorem{remark}{Remark}

\usepackage{wrapfig}
\usepackage{esint}
\usepackage[font=small]{caption}
\usepackage{todonotes} 
\usepackage{geometry}
\usepackage[title]{appendix}

\usepackage{cite}

\begin{document}

\title{Propagating terrace in a two-tubes model\\
of gravitational fingering\footnote{YP and ST have equal contribution}
}

\author{Yu. Petrova\footnote{Pontificia Universidade Catolica do Rio de Janeiro (PUC-Rio).  R. Marques de Sao Vicente, 124 - Gavea, Rio de Janeiro - RJ, 22451-040, Brazil. 
E-mail: yu.pe.petrova@gmail.com.},
S. Tikhomirov\footnote{Pontificia Universidade Catolica do Rio de Janeiro (PUC-Rio) R. Marques de Sao Vicente, 124 - Gavea, Rio de Janeiro - RJ, 22451-040, Brazil; University of Duisburg-Essen, Thea-Leymann-Strasse 9, 45127 Essen, Germany; Instituto de Matematica Pura e Aplicada, Estrada Dona Castorina, 110, Jardim Botanico, Rio de Janeiro, RJ - Brazil, CEP: 22460-320.  E-mail: sergey.tikhomirov@gmail.com.},
Ya. Efendiev\footnote{Department of Mathematics \& ISC 
Texas A\&M University. 612 Blocker Building 3404 TAMU, College Station, TX 77843-3404. E-mail:  efendiev@math.tamu.edu.}}
\renewcommand{\today}{}
\maketitle
\abstract{
We study a semi-discrete model for the two-dimensional incompressible porous medium (IPM) equation describing gravitational fingering phenomenon. The model consists of a system of advection-reaction-diffusion equations on concentration, velocity and pressure, describing motion of miscible liquids under the Darcy's law in two vertical tubes (real lines) and interflow between them. Our analysis reveals the structure of gravitational fingers in this simple setting -- the mixing zone consists of space-time regions of constant intermediate concentrations and the profile of propagation is characterized by two consecutive traveling waves which we call a terrace. We prove the existence of such a propagating terrace for the parameters corresponding to small distances between the tubes. This solution shows the possible mechanism of slowing down the fingers' growth due to convection in the transversal direction. The main tool in the proof is a reduction to pressure-free transverse flow equilibrium (TFE) model using geometrical singular perturbation theory and the persistence of stable and unstable manifolds under small perturbations.
\medskip


}

\tableofcontents

\section{Introduction}
\label{sec:intro}
\subsection{General context}

In this paper, we consider the semi-discrete model of the 2D viscous incompressible porous media (IPM) equation. The  IPM equation describes evolution of concentration carried by the flow of incompressible fluid that is determined via Darcy’s law in the field of gravity:
\begin{align}
\label{eq:IPM-1}
    \partial_t c+\mathrm{div}(uc)&=\nu\Delta c,\quad
    \\
\label{eq:IPM-2}
    \mathrm{div}(u)&=0,
    \\
\label{eq:IPM-3}
    u&=-\nabla p -(0,c).
\end{align}
Here $c = c(t,x,y)$ is the transported concentration, $u=u(t,x,y)$ is the vector field describing the fluid motion, $p=p(t,x,y)$ is the pressure,  and $\nu\geq0$ is a dimensionless parameter equal to an inverse of the P\'eclet number. The spatial domain $(x,y)$ can be either the whole space $\mathbb{R}^2$ or cylinder $[0,1]\times\mathbb{R}$ with periodic or no-flux boundary conditions. In what follows, we will consider a discretization in $x$.

There have been many recent papers analyzing the well-posedness questions for the inviscid IPM equation ($\nu=0$)~\cite{Castro-Cordoba-Lear-IPM-2019, Cordoba-Gancedo-Orive-2007}, lack of uniqueness of weak solutions~\cite{Cordoba-Faraco-Gancedo--non-uniqueness-2011, Szekelyhidi-2012} and questions of long time dynamics~\cite{Castro-Cordoba-Lear-IPM-2019, Elgindi-2017,IPM-relaxation-2024}; see also related macroscopic models~\cite{otto1999relaxation,CFG2023-macro-IPM}. The question of global regularity vs finite-time blow up is open for the IPM equation, see e.g.~\cite{2023-yao-kiselev-IPM}.

We are interested in studying the pattern formation for the initial conditions close to the unstable stratification:
\begin{align}
\label{eq-cpm1}
    c(0,x,y)=
    \begin{cases}
    +1,&y\geq0,
    \\
    -1,&y<0.
    \end{cases}
\end{align}
This corresponds to the heavier fluid on top and lighter on the bottom, and results into unstable displacement, as on Fig.~\ref{fig:fingers-Otto-Menon}a, also known as \textit{gravitational fingering} phenomenon. For this scenario and $\nu=0$ it is known that the problem is ill-posed (at least in the Muskat sense~\cite{muskat1934,ill-posed-muskat-2004, cordoba-gancedo-2007} considered as a free boundary problem). Fingering instability appears in subsurface~\cite{Groundwater-2013}, environmental \cite{Env1, Env2, Env3}, engineering \cite{Eng1}, biological \cite{Bio1, Bio2}, petroleum engineering applications \cite{TPM, Sorbie2, Sorbie3, Sorbie4} and combustion~\cite{zik1999fingering, zik1998fingering1, zik1998fingering2, kagan2008}. In particular, the related problem of  \textit{miscible viscous fingering} appears in petroleum industry  when a fluid of higher mobility displaces another of lower mobility (e.g., in the context of enhanced oil recovery (EOR) methods: polymer or gas flooding, important role plays speed of growth of the mixing zone \cite{Sorbie1, Lake, Sheng, Gao2011, GVB, GVB-SPE, CO2}).  In this case the analogue of Darcy's law~\eqref{eq:IPM-3} reads~as:
\begin{align}
\label{eq:Darcy-viscosity}
    u&=-K\cdot m(c)\nabla p.
\end{align}
Here $K$ is the porous media permeability tensor and $m(c)$ --- the mobility function. Equations \eqref{eq:IPM-1}, \eqref{eq:IPM-2}, and \eqref{eq:Darcy-viscosity} are called Peaceman model \cite{Peaceman1962}.

\begin{figure}[ht]
    \centering
    (a)\;\; 
    \includegraphics[width=0.4\textwidth]{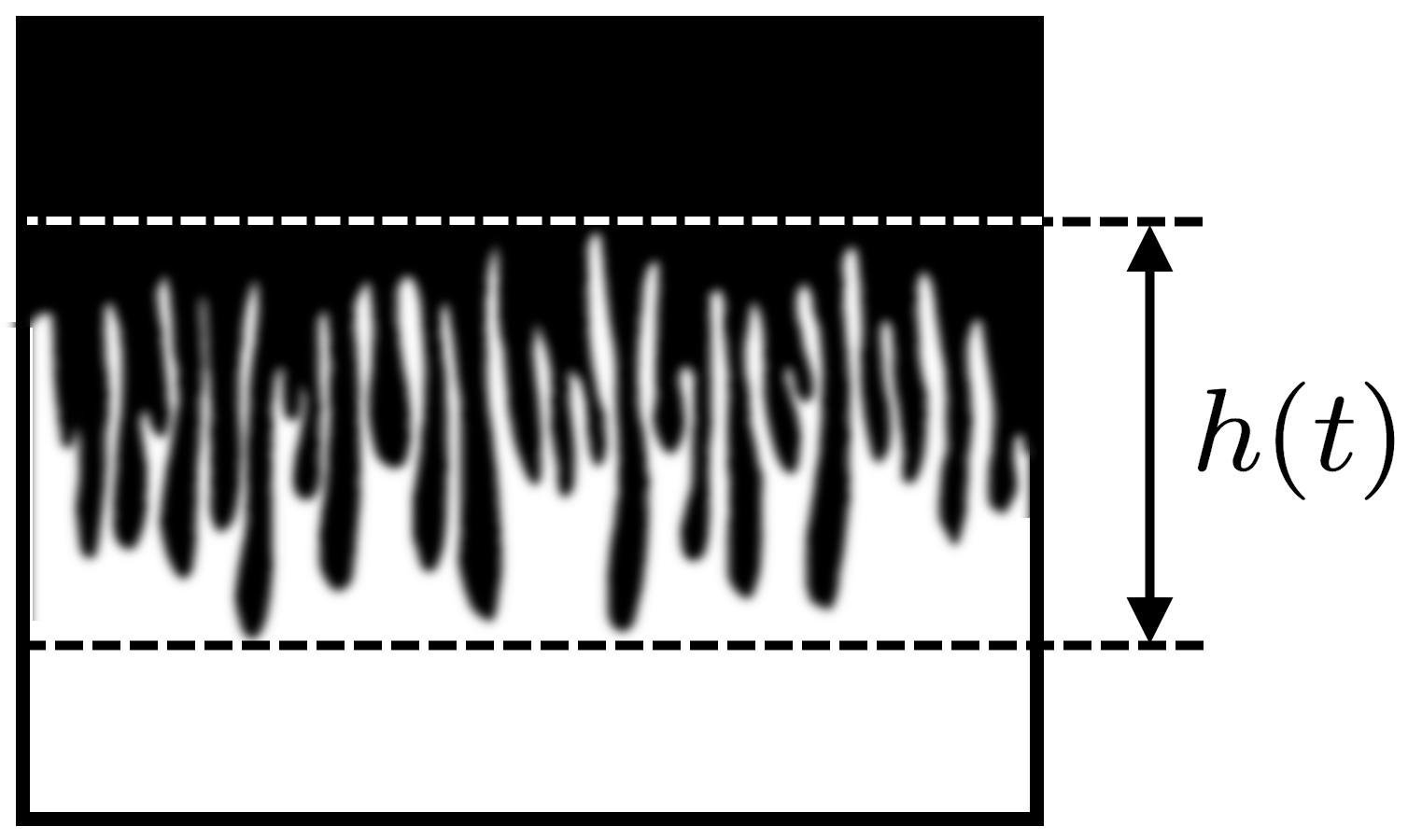}
    \hfil
    (b)
        \includegraphics[width=0.15\textwidth]{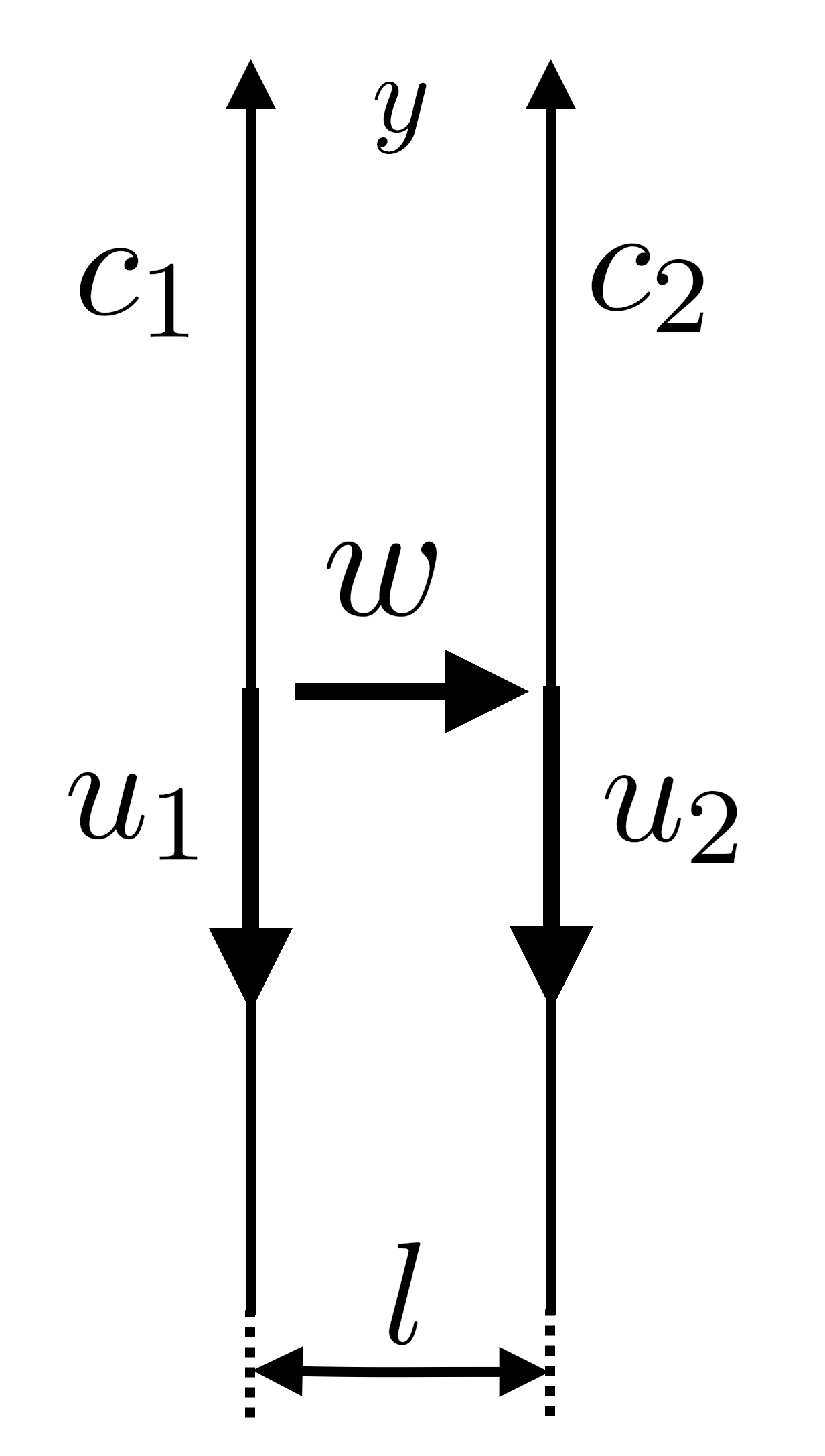}
    \caption{(a) Schematic representation of the gravitation fingering instability, $h(t)$ --- size of the mixing zone; (b) the two-tubes model. }
    \label{fig:fingers-Otto-Menon}
\end{figure}

For the inviscid IPM, this is the classical Saffman-Taylor instability. In~\cite{saffman-taylor-1958}, Saffman and Taylor discovered a one-parameter family of traveling wave solutions (which they called ``fingers''). For the viscous IPM, corresponding to miscible viscous fingers, the dynamics is distinguished by three regimes: (i) at early times, the flow is well described by linear stability theory; (ii) at intermediate times, the flow is dominated by nonlinear finger interactions which evolve into a mixing zone; and (iii)~at late times the flow regime consists of pair of counter-propagating fingers \cite{nijjer-2018}. 

Many laboratory and numerical experiments show
the linear growth of the mixing zone for moderate time regimes (viscous fingering: e.g. \cite{nijjer-2018, JCAM-Comp, 2023-intermediate-Tikhomirov}, gravitational fingering: \cite{MH-95, Wooding-1969, Boffetta1, Boffetta2, PRF2019}; see also surveys \cite{Homsy-1987, Survey-2017}).  A theoretical approach\footnote{In most of the results the initial data was considered to be  $c\in[0, 1]$; in what follows we provide all estimates scaled to initial data~\eqref{eq-cpm1}.} 
to estimate the width  $h(t)$ of the mixing zone  for the system~\eqref{eq:IPM-1}--\eqref{eq:IPM-3} was done in~\cite{otto-menon-2005}.  In particular, the authors get $h(t)\leq 4t$ using energy estimates, and provide pointwise estimates 
\begin{equation}\label{eqat}
    h(t)\leq 2t 
\end{equation}
for a reduced model that we call Transverse Flow Equilibrium (TFE; for generalisations see~\cite{otto-menon-2006, yortsos-salin-2006}) which consists of equations~\eqref{eq:IPM-1}--\eqref{eq:IPM-2} and the following expression for $u$ instead of the Darcy's law 
(here $(x,y)\in[0,1]\times\mathbb{R}$):
\begin{align}
\label{eq:TFE-2D-u}
    u=(u^x,u^y),\qquad u^y=\bar{c}-c, \qquad \bar{c}(t,y)=
    \int_{0}^1 c(t,x,y)\,dx.
\end{align}
Quantification of the size of the mixing zone in laboratory and numerical experiments for two-dimensional space does not give exact answer: it shows that (see \cite{MH-95, Wooding-1969, Survey-2017, Boffetta2})
\begin{align}
\label{eq:al1.34}
h(t) \sim \alpha t, \quad \mbox{for some $\alpha \in [1.34, 2]$}.   
\end{align}
For 3D IPM model the size of mixing zone is claimed to be smaller $\alpha = 0.86$ (see \cite{Boffetta2}).

Our motivation is to understand the mechanism of finger formation and sharpen existing estimates for the size of the mixing zone. We show in a simplified setting that there are two (interconnected) mechanisms of possible slowing down of the fingers, thus, a decrease in the size of the mixing zone: (1)~the convection in the transverse direction of the flow; (2) the presence of the intermediate concentration, that is the typical concentration inside the finger is $c^*\in(-1,1)$. In particular, our results suggest that estimates \eqref{eqat} could be improved. The difference between values $1.34$ and $2$ in \eqref{eq:al1.34} is not negligible and its clarification is needed. Additionally it was shown in~\cite{2023-intermediate-Tikhomirov} that, for the case of viscous fingers, intermediate concentration could lead to even more significant slowdown. See Appendix~\ref{ap:slowdown} for a more detailed explanation on what we mean by ``slowdown of fingers due to intermediate concentration''.

\subsection{Informal description of the model and main results}
In this paper, we introduce a new model for the gravitational fingering motivated by a finite-volume scheme with a simple upwind, that is widely used in multiphase flow simulations~\cite{Finite-volume-2000}. Our basic assumption is that the fluid displacement happens inside two vertical lines (``tubes''), and the transverse flow between the tubes is governed by the discrete Darcy's law: the velocity is equal to the pressure difference divided by the distance between the tubes. This results in a system of nonlinear reaction-diffusion-convection equations that we call the~\textit{two-tubes IPM~equations}.

More precisely, consider the two tubes denoted by 1 and 2. The model we study reads (see Fig.~\ref{fig:fingers-Otto-Menon}b; $y\in\mathbb{R}$, after rescaling we can take $\nu=1$)
\begin{align*}
    \partial_t c_1+\partial_y(u_1c_1)- \partial_{yy}c_1&=-f,
    \\
    \partial_t c_2+\partial_y(u_2c_2)- \partial_{yy}c_2&=f.
\end{align*}
Here $c_{1,2}=c_{1,2}(t,y)$ are the concentrations and $u_{1,2}=u_{1,2}(t,y)$ are the velocities in tubes 1 and 2. Meanwhile the function $f$ is responsible for the flow between the tubes and is defined by
\begin{align*}
    f=\begin{cases}
    \cfrac{u_T}{l} \cdot c_1,& u_T\geq 0,\quad \text{(fluid flows from tube 1 to 2)},\\
    \cfrac{u_T}{l} \cdot c_2,&u_T\leq 0,\quad \text{(fluid flows from tube 2 to 1)}.
    \end{cases}
\end{align*}
Here $u_T=u_T(t,y)$ is the velocity of the fluid in the transversal direction to the tubes ($T$ stands for ``Transversal''). Note that parameter $l>0$ can be interpreted as a distance between the tubes (or the cell size in $x$-direction in a finite-volume scheme); it plays a crucial role in our analysis. The system of equations becomes closed  if we supply it with expressions for the fluid velocities $u_{1,2,T}$, analogous to equations~\eqref{eq:IPM-2}--\eqref{eq:IPM-3}, see Section~\ref{sec:problem-statement} for the complete description and Appendix~\ref{ap:FV} for the schematic derivation of the model from the finite-volume scheme. 

\begin{figure}[ht]
    \centering
    (a)\quad
    \includegraphics[width=0.39\textwidth]{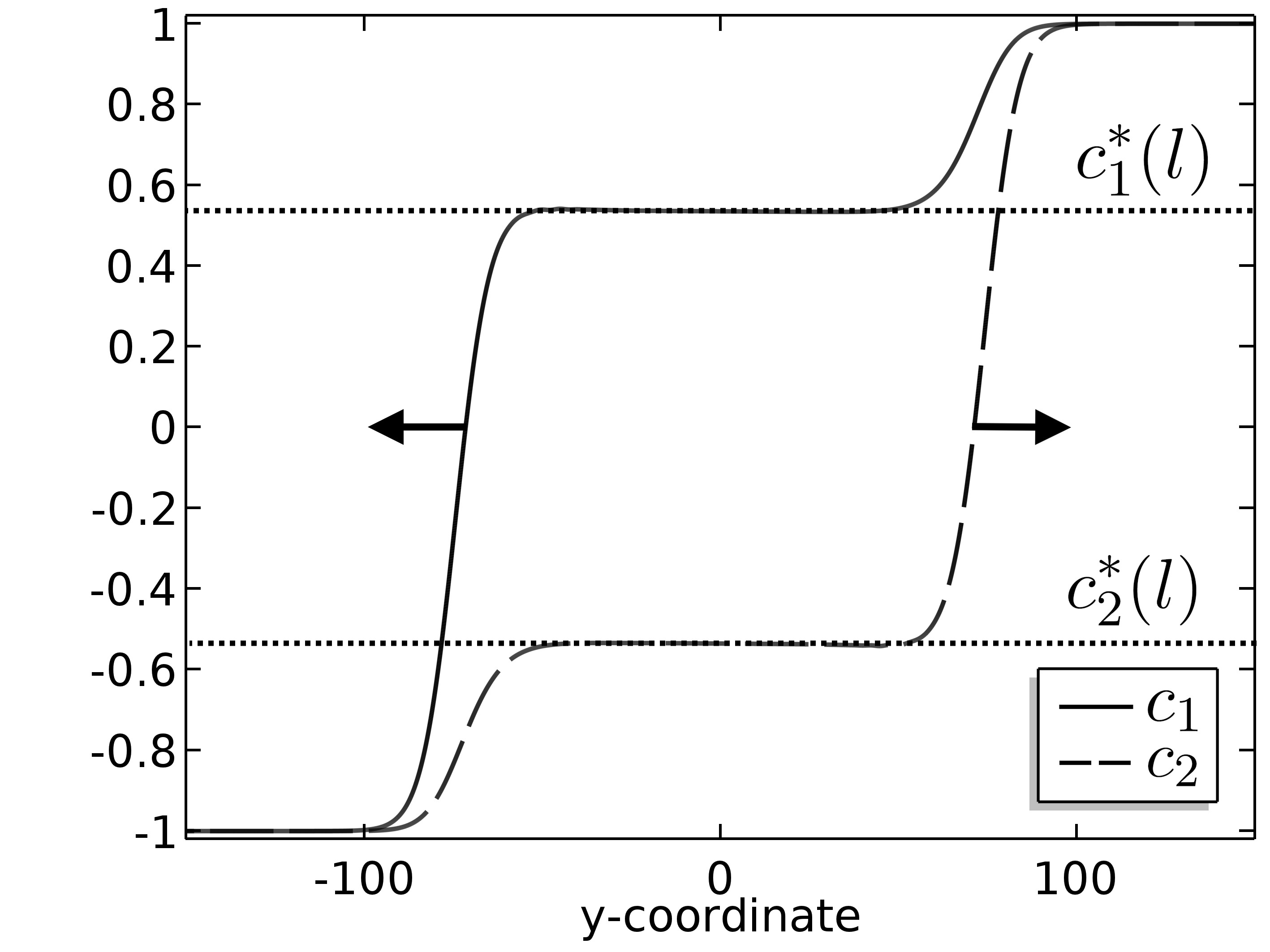}
    \hfil
    (b)\quad 
    \includegraphics[width=0.26\textwidth]{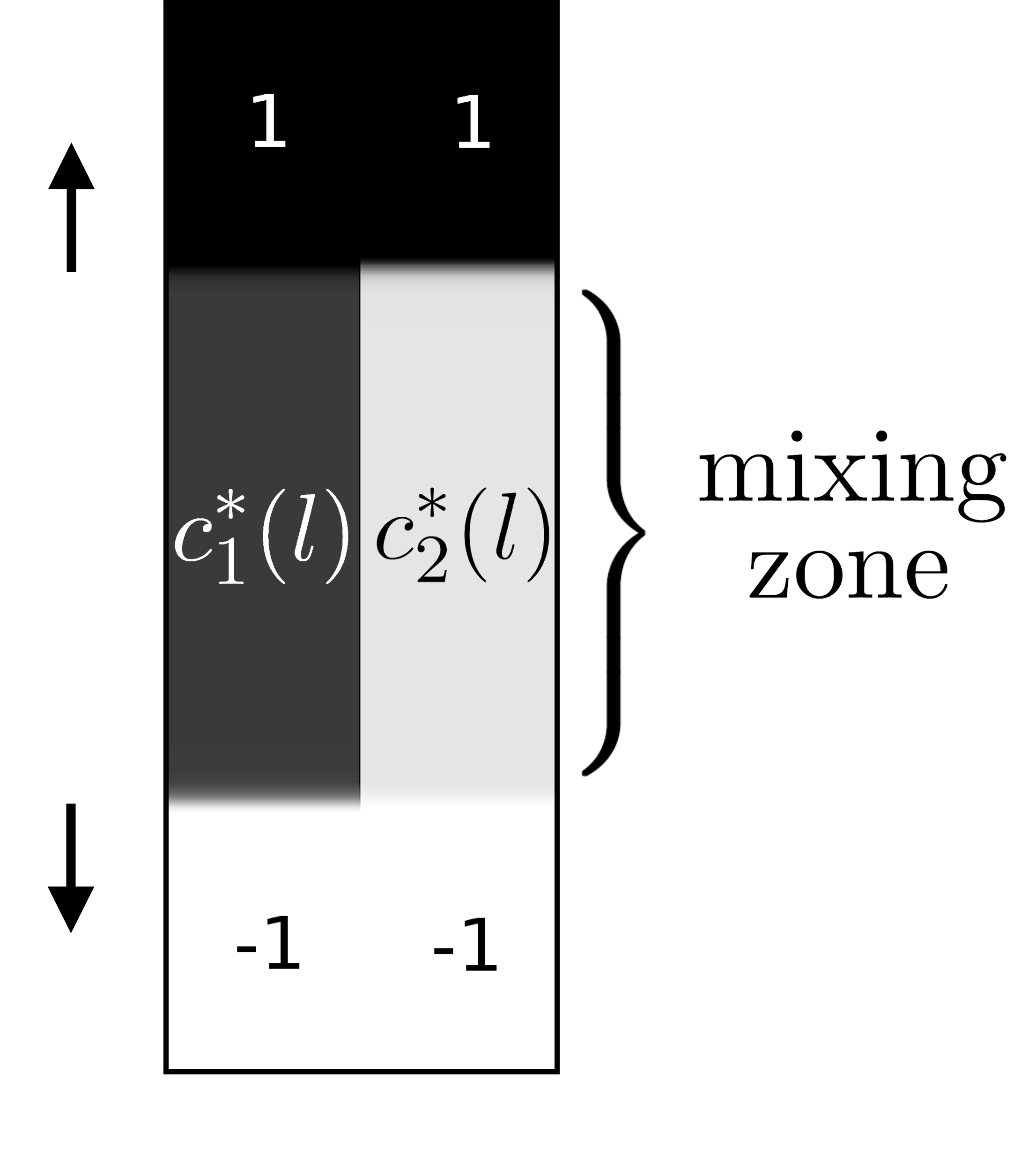}
    \caption{(a) Numerical evidence of propagating terrace consisting of two traveling waves; (b) Schematic representation of propagating terrace as two counter-propagating fingers.}
    \label{fig:cascades}
\end{figure}

The main result (Theorem~\ref{thm:main-Peaceman}) claims that for sufficiently small values of $l>0$ there exist two intermediate concentrations~$c_1^*(l)\in(-1,1)$, $c_2^*(l)\in(-1,1)$ and two traveling wave (TW) solutions that connect the states:
\begin{align*}
    (-1,-1) \xrightarrow{TW} (c_1^*(l),c_2^*(l))\xrightarrow{TW} (1,1).
\end{align*}
Moreover, the speeds of the traveling waves approach $-1/4$ and $1/4$ as $l\to0$. 
The numerical modelling of the two-tubes IPM model, see Fig.~\ref{fig:cascades}a,  shows that such combination of two traveling waves (we will call it a \textit{propagating terrace}) appears as a long-term limit for typical initial data with proper values at infinity.

The appearing propagating terrace allows for the following ``physical'' interpretation. The region between the two traveling waves corresponds to the mixing zone which consists of the two counter-propagating fingers: one with the constant concentration $c_1^*(l)$, the other ---  $c_2^*(l)$, see Fig.~\ref{fig:cascades}b. Notice that we get the linear growth of the mixing zone for granted, since the boundary of the mixing zone is a traveling wave moving at a constant speed.

Roughly speaking the proof of Theorem~\ref{thm:main-Peaceman} is based on the perturbation argument and valid only for small values of parameter $l>0$. A natural question to ask~is
$$
 \mbox{``What system of equations do we get in a singular limit $l=0$? (at least formally)''}
$$
 
We call this system \textit{the two-tubes Transverse Flow Equilibrium model} (see~\eqref{eq:2tubes-diffusive-system-1}--\eqref{eq:peretok},~\eqref{eq:IPM-2-tubes-w},~\eqref{eq:gravitational-Darcy-TFE-2tubes}) as it turns out to be a two-tubes analogue of the above mentioned 2D TFE model~\eqref{eq:IPM-1}--\eqref{eq:IPM-2},~\eqref{eq:TFE-2D-u}. The intuitive idea behind TFE model is that a certain equilibrium is assumed in the transverse direction to the gravity (for the case of viscous fingering see e.g.~\cite{armiti-rohde-2014,armiti-rohde-2019},~\cite{yortsos1995-TFE}\footnote{Yortsos~\cite{yortsos1995-TFE} called it vertical equilibrium (VE)} and references therein). 
In other words the ``fluid flows (or equilibrates) much faster in the transverse direction than to the main direction of the flow''. Mathematically, TFE model can be derived from the assumption that the pressure $p(t,x,y)\approx p(t,y)$ doesn't depend on $x$-variable up to the first approximation. This situation occurs, for example, in an asymptotically thin domain. Intuitively this asymptotic limit for the two-tubes model corresponds to  $l\to0$. See Appendix~\ref{ap:TFE-IPM} for more precise argument on the connection between IPM and TFE models.  Note that 2D TFE model also exhibit a fingering instability and has both theoretical advantages~\cite[Section 2.3]{otto-menon-2006} and remarkable reduction in computational complexity~\cite{armiti-rohde-2019}, and thus has its own value and interest.

The second main result (Theorem~\ref{thm:main-TFE}) establishes  the existence and uniqueness of a propagating terrace for the two-tubes TFE model. Surprisingly, in  Theorem~\ref{thm:main-TFE} we managed to find explicitly the trajectories of the traveling wave dynamical system. 
Note that each of the traveling waves  connecting 
$$(-1,-1) \xrightarrow{TW} (c_1^*,c_2^*) 
\quad \text{or}\quad  
(c_1^*,c_2^*)\xrightarrow{TW} (1,1)
$$ 
exists for a one-parameter family of $(c_1^*,c_2^*)$,  but the propagating terrace 
$$
(-1,-1) \xrightarrow{TW} (c_1^*,c_2^*)\xrightarrow{TW} (1,1)
$$
is unique.

Let us emphasize an important relation between the proofs of Theorem~\ref{thm:main-Peaceman} and Theorem~\ref{thm:main-TFE}.  Traveling wave dynamical system corresponding to the TFE model  is  a formal singular limit as $l \to 0$ of traveling wave dynamical system corresponding to the IPM model. Using geometric singular perturbation theory  and invariant manifolds it is possible to extend properties of heteroclinic trajectories from the TFE system to the IPM system with sufficiently small $l>0$. See the end of Section~\ref{sec:problem-statement} for details.

Propagating fronts in the layered media appear in different contexts, to mention a few: multi-lane traffic flow~\cite{multi-lane-1998,2019-Holden-Risebro, 2024-goatin}, combustion~\cite{schecter-2006-two-layers}, foam flow~\cite{foam-2022-two-layers}, two-phase flow in porous media in layers with different properties~\cite{debbabi2017} (permeabilities, porosities, thermal conductivities etc), and voltage conduction along the nerve fibers~\cite{bose-1995}.

The novel model we present here is natural and simple, illustrates the possible mechanism of the reducing the speed of fingers due to interflow between the tubes, while at the same time allowing for a rigorous mathematical treatment.

The paper is organised as follows: in Section~\ref{sec:problem-statement} we formulate the two-tubes model (for IPM and TFE) and state the main results --- Theorem~\ref{thm:main-Peaceman} and Theorem~\ref{thm:main-TFE} --- the existence of the propagating terrace consisting of two traveling waves. Also we provide a sketch of their proofs. In Section~\ref{sec:proof-thm2} we give a 
detailed analysis of the two-tubes TFE model: derivation and properties of the traveling wave dynamical system, study the existence of traveling waves (Theorem~\ref{thm:TW-TFE}) and propagating terrace (Theorem~\ref{thm:main-TFE}). In Section~\ref{sec:proof-peaceman} we give a detailed analysis of the two-tubes IPM model: derivation and properties of the traveling wave dynamical system,  existence of traveling waves (Theorem~\ref{thm:TW-IPM}) and propagating terrace (Theorem~\ref{thm:main-Peaceman}).  Section~\ref{sec:discussions} is devoted to discussions and possible generalizations. In Appendix~\ref{ap:FV} we give a schematic derivation of the presented models from the upwind finite-volume scheme. Appendix~\ref{ap:TFE-IPM} is devoted to description of the  connection between IPM and TFE models. In Appendix~\ref{ap:slowdown} we discuss the term ``slowdown of fingers due to intermediate concentration''.

\section{Problem statement and main result}
\label{sec:problem-statement}
Our aim  is to consider the simple setting in which we can analyze the impact of the flow in transverse direction (with respect to gravity) on the speed of propagation of the mixing zone.  The idea is to consider IPM equations~\eqref{eq:IPM-1}--\eqref{eq:IPM-3} and discretize the space in $x$ direction. Note that similar idea was proposed in~\cite[Section 3]{yortsos1995-TFE} and studied numerically in~\cite[Section 3.4]{armiti-rohde-2019} in the context of the two-phase flows in porous media.

In this paper, we focus on 
the simplest possible formulation --- the case of two tubes. The $n$-tubes model, $n>2$, is a subject for future work, some heuristic results can be found  in~\cite{Oberwolfach-2024}. Let $c_{1,2}(t,y)$ be the concentrations and $u_{1,2}(t,y)$ be the velocities in tube 1 and 2. Then the governing equations of the flow inside each tube are:
\begin{align}
\label{eq:2tubes-diffusive-system-1}
    \partial_t c_1+\partial_y(u_1c_1)- \partial_{yy}c_1&=-f,
    \\
    \label{eq:2tubes-diffusive-system-2}
    \partial_t c_2+\partial_y(u_2c_2)- \partial_{yy}c_2&=f.
\end{align}
The flow $f$ from tube 1 to tube 2 equals:
\begin{align}
\label{eq:peretok}
    f=\begin{cases}
    \cfrac{u_T}{l} \cdot c_1,& u_T\geq 0,\\
    \cfrac{u_T}{l} \cdot c_2,& u_T\leq 0.
    \end{cases}
\end{align}
Here $u_T=u_T(t,y)$ is the velocity of the fluid in the transversal direction to the tubes.  
In what follows we will consider two models for the velocities $u_{1,2,T}$ that correspond to IPM and TFE models.

\subsection{Two-tubes IPM model: definition and main theorem}
 Let $p_{1,2}(t,y)$ be the pressures in tubes 1 and 2. The  
 analogue of the Darcy's law~\eqref{eq:IPM-3}, is the following relation between the pressure and velocity:
\begin{align}
\label{eq:gravitational-Darcy-2tubes-a}
    &\text{(Darcy's law in tube 1 and 2)}& u_1& = -\partial_y p_1 -c_1,&
    u_2& = -\partial_y p_2 -c_2,
    \\
    \label{eq:gravitational-Darcy-2tubes-b}
    &\text{(Darcy's law between tubes)}
    & u_T &= \frac{p_1-p_2}{l}.
\end{align}
Here $l>0$ is a parameter that corresponds to the distance between the tubes. The semi-discrete analogue of condition~\eqref{eq:IPM-2} writes as:
\begin{align}
\label{eq:IPM-2-tubes-w}
    \cfrac{u_T}{l}+\partial_y u_1=0, \quad -\cfrac{u_T}{l} + \partial_y u_2 = 0.
\end{align}
We call the system~\eqref{eq:2tubes-diffusive-system-1}--\eqref{eq:IPM-2-tubes-w} the \textit{two-tubes IPM equations}.

Instead of considering the pressures $p_1$ and $p_2$ individually, the natural variable to consider is the pressure drop $q:=p_2-p_1$. Thus, the equations~\eqref{eq:gravitational-Darcy-2tubes-a}--\eqref{eq:IPM-2-tubes-w} can be simplified
\begin{align}
\label{eq:gravitational-Darcy-2tubes-c}
    \partial_y q& = u_1 - u_2 + c_1 - c_2,&
    \partial_y u_1 &= \cfrac{q}{l^2}=-\cfrac{u_T}{l}.
\end{align}
Moreover, relations~\eqref{eq:IPM-2-tubes-w} imply $\partial_y(u_1+u_2)=0$. As we assume later the conditions~\eqref{eq:infty-cond}, we deduce 
\begin{align}
\label{eq:u1+u2=0}
    u_1+u_2\equiv0.    
\end{align}
One can think of the model~\eqref{eq:2tubes-diffusive-system-1}--\eqref{eq:IPM-2-tubes-w} as a limit of the finite-volume discretization of the model~\eqref{eq:IPM-1}--\eqref{eq:IPM-3} with spatial grid $2\times \mathbb{Z}$ of cells having sizes $l\times h$ as $h\to0$, see Appendix A for details. Equation~\eqref{eq:peretok} corresponds to upwind scheme.
In Appendix~\ref{ap:TFE-IPM} we give a comment about the analogue of the parameter $l$ for the 2D IPM model.

We are interested in finding solutions $(c_1,c_2,u_1,u_2,q)$ of~\eqref{eq:2tubes-diffusive-system-1}--\eqref{eq:peretok},~\eqref{eq:gravitational-Darcy-2tubes-c},~\eqref{eq:u1+u2=0} satisfying the following conditions at $\pm\infty$ (physical meaning --- heavier fluid up, lighter fluid down; no flow inside the tubes and between them at $\pm\infty$):
\begin{align}
\label{eq:infty-cond}
    c_{1,2}(-\infty)=-1,\qquad c_{1,2}(+\infty)=1,\qquad q(\pm\infty)=0, \qquad u_{1,2}(\pm\infty)=0.\qquad
\end{align}

Recall that a function $g(t,y):\mathbb{R}_+\times\mathbb{R}\to\mathbb{R}^n$, $n\in\mathbb{N}$, is a \textit{traveling wave}\cite{gilding2004travelling, smoller2012shock} with speed~$v\in\mathbb{R}$ connecting states $g_-\in\mathbb{R}^n$ and $g_+\in\mathbb{R}^n$, if it has the form $g(t,y)=\tilde{g}(y-vt)$, where $\tilde{g}:\mathbb{R}\to\mathbb{R}^n$ is a continuous function, which satisfies $\tilde{g}(-\infty)=g_-$, $\tilde{g}(+\infty)=g_+$.

Numerically, we observe convergence of the solution of~\eqref{eq:2tubes-diffusive-system-1}--\eqref{eq:peretok},~\eqref{eq:gravitational-Darcy-2tubes-c}--\eqref{eq:infty-cond} to a combination of two traveling waves.  Following~\cite{DGM-terraces-2014} that describes fronts in reaction-diffusion equations (see also~\cite{polacik2017,polacik2020, giletti2020,giletti2022}), we introduce a notion of \textit{propagating terrace} for our problem.

\begin{definition}
\label{def:propag-terrace}
    A \textbf{propagating terrace} connecting $\alpha\in\mathbb{R}^5$ to $\beta\in\mathbb{R}^5$ is a pair of finite sequences $(\sigma_k)_{0\leq k\leq N}$ and $(g_k)_{1\leq k\leq N}$ such that:
    \begin{itemize}
        \item Each $\sigma_k=(c_{1k},c_{2k},u_{1k},u_{2k},q_k)$ is a stationary solution  of~\eqref{eq:2tubes-diffusive-system-1}--\eqref{eq:peretok},~\eqref{eq:gravitational-Darcy-2tubes-c}, \eqref{eq:u1+u2=0} and
        $\sigma_0=\alpha$, $\sigma_N=\beta$.
        \item $g_k$ is a traveling wave solution of~\eqref{eq:2tubes-diffusive-system-1}--\eqref{eq:peretok},~\eqref{eq:gravitational-Darcy-2tubes-c},~\eqref{eq:u1+u2=0} connecting $\sigma_{k-1}$ to $\sigma_{k}$, $1\leq k\leq N$.
        \item The speed $v_k\in\mathbb{R}$ of each $g_k$ satisfies $v_1\leq v_2\leq\ldots\leq v_N$.
    \end{itemize}
\end{definition}

The following theorem is the first main result of the paper.

\begin{theorem}
\label{thm:main-Peaceman}
Consider the two-tubes IPM equations~\eqref{eq:2tubes-diffusive-system-1}--\eqref{eq:peretok},~\eqref{eq:gravitational-Darcy-2tubes-c},~\eqref{eq:u1+u2=0}. Then there exists  sufficiently small $l_0>0$ such that for all $l\in(0,l_0)$ there exist a propagating terrace of two traveling waves with speeds $v_1^*(l)$, $v_2^*(l)$ connecting the states 
\begin{align*}
  \sigma_0=(-1,-1,0,0,0),\quad \sigma_1=(c_1^*(l),c_2^*(l), u_1^*(l),u_2^*(l),0), \quad \sigma_2=(1,1,0,0,0).
\end{align*}

Moreover, as $l\to0$ we obtain:
\begin{align*}
        &\lim\limits_{l\to 0} c_1^*(l)=-\lim\limits_{l\to 0} c_2^*(l)=1/2, &
        \lim\limits_{l\to 0} v_2^*(l)=-\lim\limits_{l\to 0} v_1^*(l)=1/4.
\end{align*}
\end{theorem}

\begin{remark}\label{rem:pt-names}

    In the present paper we study existence of propagating terrace. For stability analysis of propagating terraces in different contexts see for instance mentioned above papers \cite{DGM-terraces-2014, polacik2017,polacik2020}.  In literature there are several similar notions known by different names: ``stacked'' combination of wave fronts (or ``minimal decomposition'') in the classical paper by Fife and McLeod~\cite{fife-mcleod-1977} (see also~\cite{roquejoffre1996,iida2011stacked}); wave trains~\cite{feinberg1991,roquejoffre1996,volpert1996}; ordered systems of waves~\cite{volpert2001-systemsofwaves}; minimal systems of waves~\cite[Sect.~I.1.5]{3volpert1994-book},~\cite{mejai1999}; concatenated traveling waves \cite{lin2015stability,lin2016stability};  multifronts or multipulses \cite{Wright2009, BST2008}. In the cited papers not only the questions of existence, but also the stability are studied.
\end{remark}

The proof of Theorem~\ref{thm:main-Peaceman} is based on the fine analysis of the auxiliary system that we call two-tubes TFE model and describe below.

\subsection{Two-tubes TFE model: definition and main theorem}
\label{subsec:TFE-def-thm}
Consider the so-called Transverse Flow Equilibrium (TFE) model for velocities:
\begin{align}
\label{eq:gravitational-Darcy-TFE-2tubes}
     u_1=(c_2-c_1)/2;&& u_2=-u_1.
\end{align}
This is a discrete analogue of the two-dimensional model originally introduced in~\cite{Wooding-1969}, and analyzed in~\cite{otto-menon-2005}. We call the system~\eqref{eq:2tubes-diffusive-system-1}--\eqref{eq:peretok}, \eqref{eq:IPM-2-tubes-w}, \eqref{eq:gravitational-Darcy-TFE-2tubes} the \textit{two-tubes TFE equations}. In a similar way to Definition~\ref{def:propag-terrace} the notion of the propagating terrace can be introduced for the two-tubes TFE equations. In this case  $\sigma_k=(c_{1k},c_{2k}),\alpha,\beta\in\mathbb{R}^2$. 

The second principal result of the paper is stated as follows.

\begin{theorem}
\label{thm:main-TFE} 
Consider the two-tubes TFE equations~\eqref{eq:2tubes-diffusive-system-1}--\eqref{eq:peretok},~\eqref{eq:IPM-2-tubes-w},~\eqref{eq:gravitational-Darcy-TFE-2tubes}. Then there exist exactly two propagating terraces of two traveling waves with speeds 
$$v_1^*=-1/4, \quad v_2^*=1/4$$ 
connecting the states
\begin{align*}
  \sigma_0=(-1,-1),\quad  \sigma_1=(1/2,-1/2) \;(\text{or } \sigma_1=(-1/2,1/2)), \quad \sigma_2=(1,1).  
\end{align*}
\end{theorem}
See Remark~\ref{rmk:uniqueness} explaining an obstruction to prove uniqueness of the traveling wave solutions in Theorem~\ref{thm:main-Peaceman} in contrast to Theorem~\ref{thm:main-TFE}.

Theorem~\ref{thm:main-TFE} can be interpreted in the following way: the width of the mixing zone for the two-tubes TFE equations is $h(t)\sim \frac12t$, which means a slowdown compared to estimates~\eqref{eqat}. We believe that the interflow between the tubes and the presence of the intermediate concentration are responsible for such behavior. For more details see Appendix~\ref{ap:slowdown}.

\subsection{Scheme of proof of the main theorems}
The logical framework of our analysis for the two-tubes TFE and IPM models is as follows:
\begin{itemize}
    \item[Step\phantom{I} I:]  
    \textit{Study the structure of the set of traveling wave (TW) solutions.} 
    \item[Step II:] \textit{Study the existence of propagating terraces consisting of two traveling waves.} 
\end{itemize}

These two steps are performed differently for the two-tubes TFE and IPM models. For the TFE case we provide explicit constructions. For the IPM model we use a perturbation argument valid only for small enough $l>0$. Note that in our reasoning TFE model plays a role of a singular limit for IPM as $l\to0$, thus we start the discussion with the two-tubes TFE model.

\subsubsection{Two-tubes TFE model}
We give a detailed analysis of the two-tubes TFE model in Section~\ref{sec:proof-thm2}. Step I is done in Section~\ref{sec:TW-TFE-main}; step II --- in Section~\ref{subsec:prop-terrace-TFE}.

\smallskip
\begin{itemize}
    \item[Step I.] \textit{Traveling wave solutions.} The main result states as follows.
\end{itemize}
    
    \begin{theorem} 
    \label{thm:TW-TFE}
    Consider the two-tubes TFE equations~\eqref{eq:2tubes-diffusive-system-1}--\eqref{eq:peretok},~\eqref{eq:IPM-2-tubes-w},~\eqref{eq:gravitational-Darcy-TFE-2tubes}.
    \begin{itemize}
        \item[A)] For every $v < 0$ there exist exactly two traveling wave solutions $c_{1,2}(y-vt)$ satisfying $(c_1,c_2)(-\infty)=(-1,-1)$. They connect the states:
        $$
        (-1,-1) \xrightarrow{TW} (-2v-1,-6v-1) 
        \quad
        \text{and}\quad
        (-1,-1) \xrightarrow{TW} (-6v-1,-2v-1).
        $$

        For $v = 0$ the only stationary 
        solution $c_{1,2}(y)$ satisfying $(c_1,c_2)(-\infty)=(-1,-1)$ is constant $(c_1,c_2)\equiv(-1,-1)$.
        
        For $v > 0$ there are no traveling waves satisfying $(c_{1},c_2)(-\infty)=(-1,-1)$.

        Denote all the possible pairs of states $(c_1,c_2)(+\infty)$ as $\mathcal{H}_{-1}$, see Fig.~\ref{fig:HugoniotTFE}.

        \item[B)] For every $v > 0$ there exist exactly two traveling wave solutions $c_{1,2}(y-vt)$ satisfying $(c_1,c_2)(+\infty)=(1,1)$. They connect the states:
        $$
        (-2v+1,-6v+1) 
        \xrightarrow{TW}
        (1,1)
        \quad
        \text{and}
        \quad
        (-6v+1,-2v+1) \xrightarrow{TW}(1,1).
        $$

        For $v = 0$ the only stationary 
        solution $c_{1,2}(y)$ satisfying $(c_1,c_2)(+\infty)=(1,1)$ is constant $(c_1,c_2)\equiv(1,1)$.

        For $v < 0$ there are no traveling waves satisfying $(c_{1},c_2)(+\infty)=(1,1)$.

        Denote all the possible pairs of states $(c_1,c_2)(-\infty)$ as $\mathcal{H}_{1}$, see Fig.~\ref{fig:HugoniotTFE}.
 \end{itemize}
       
    \end{theorem}

    To prove Theorem~\ref{thm:TW-TFE}  we put a traveling wave ansatz into the two-tubes TFE model and obtain a family of dynamical systems parametrized by the speed of the traveling wave $v\in\mathbb{R}$. We focus on finding the heteroclinic trajectories of these dynamical systems as they correspond to the traveling wave solutions. In particular, in Section~\ref{sec:TW-TFE-main} we derive a four-dimensional dynamical system, for which we manage to find explicitly two-dimensional invariant manifolds (Proposition~\ref{rm:inv-manifolds}) and exact heteroclinic trajectories inside these manifolds (Lemma~\ref{lm:inside-inv-man}) for all possible values of the parameter $v\in\mathbb{R}$.  Theorem~\ref{thm:TFE-reformulated} states the key result on the existence of the heteroclinic trajectories and implies Theorem~\ref{thm:TW-TFE}. 
    
\smallskip

\begin{minipage}{0.37\textwidth}
\begin{center}
 \includegraphics[width=0.69\textwidth]{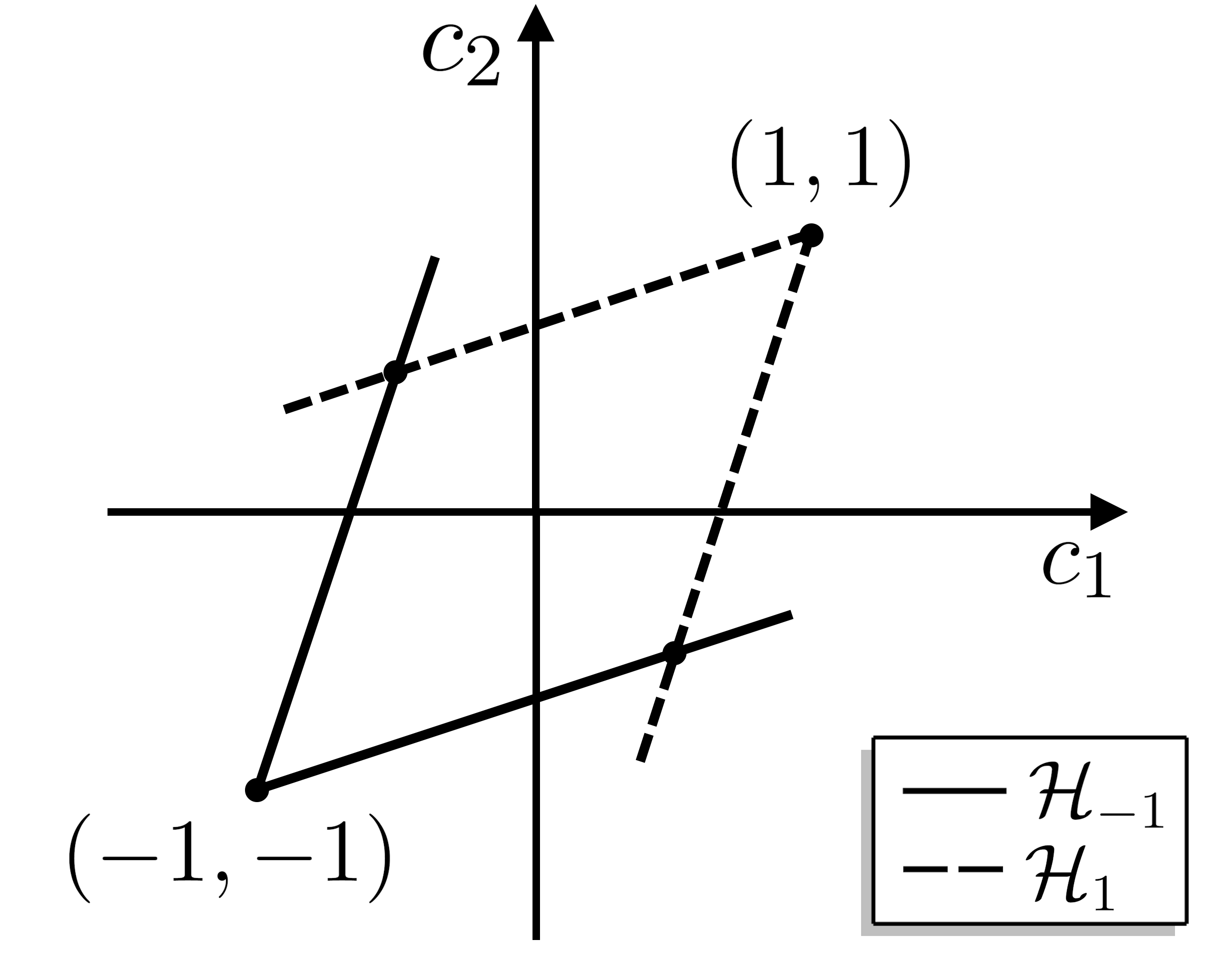}
 \captionof{figure}{Sets $\mathcal{H}_{\pm 1}$ in $(c_1,c_2)$-plane for two-tubes TFE model.} \label{fig:HugoniotTFE}
   
\end{center}
\end{minipage}
\hfill
\begin{minipage}{0.484\textwidth}
 \includegraphics[width=\textwidth]{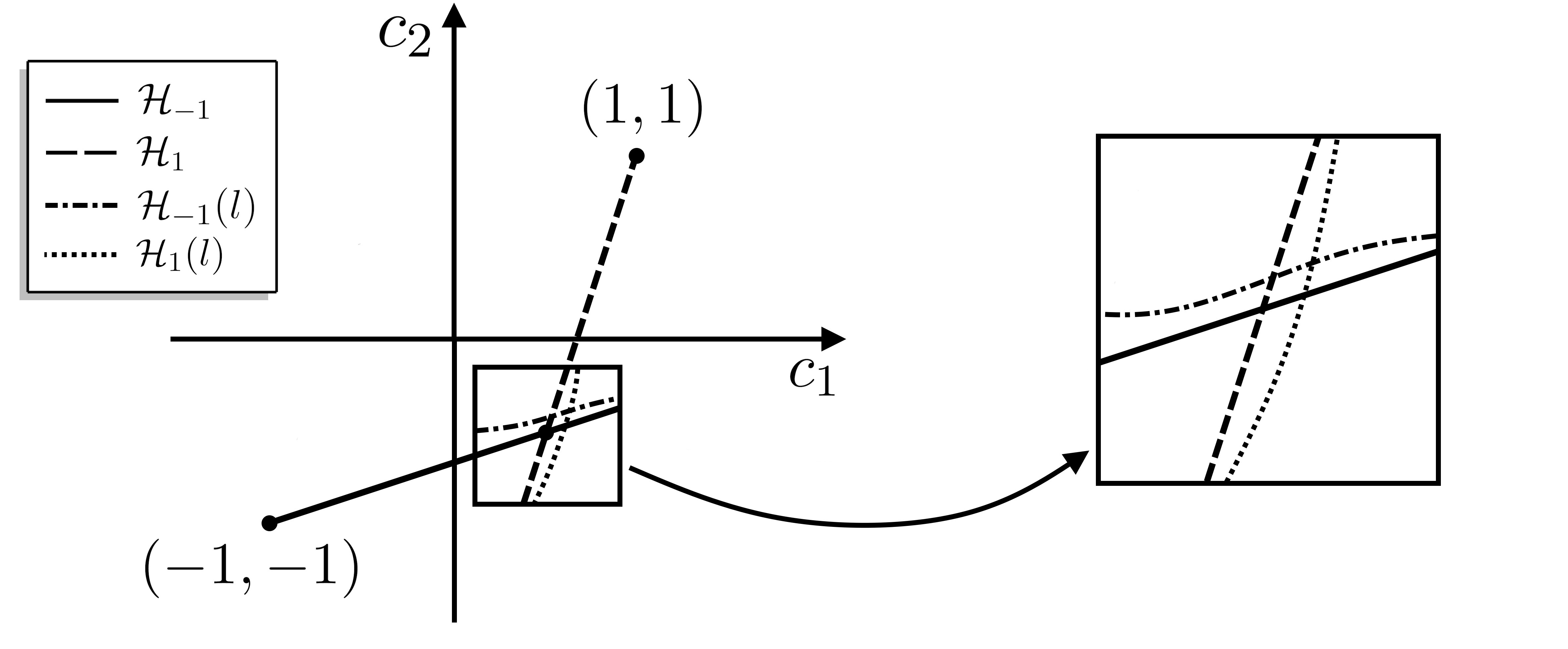}
 \captionof{figure}{Sets $\mathcal{H}_{\pm 1}(l)$ in $(c_1,c_2)$-plane for two-tubes IPM model.}  \label{fig:HugoniotIPM}
\end{minipage}
    
\begin{itemize}
    \item[Step II.] \textit{Propagating terrace.}
\end{itemize}

By Step I, we have a classification of the traveling wave solutions satisfying $(c_1,c_2)(-\infty)=(-1,-1)$ or $(c_1,c_2)(+\infty)=(1,1)$. In order to find propagation terraces, in Section~\ref{subsec:prop-terrace-TFE}, we note that the intersection $\mathcal{H}_1\cap\mathcal{H}_{-1}$ consists of two points $(-\frac12,\frac12)$ and $(\frac12,-\frac12)$, and provides the existence and uniqueness of the propagation terraces from the Theorem~\ref{thm:main-TFE}. See Fig.~\ref{fig:HugoniotTFE}.

 In what follows  we will need a useful observation proved in  Section~\ref{subsec:transverse-TFE}: the heteroclinic orbits that correspond to the traveling wave solutions can be represented as a transversal intersection of some stable and unstable manifolds (thus, persist under small perturbations). This fact is not essential for the two-tubes TFE model, but will be used in the proof of the Theorem~\ref{thm:main-Peaceman} for the two-tubes IPM model.

\subsubsection{Two-tubes IPM model}
We give a detailed analysis of the two-tubes IPM model in Section~\ref{sec:proof-peaceman}. Step I is done in Section~\ref{subsec:Peaceman-TW-main}; step II --- in Section~\ref{subsec:prop-terrace-IPM}. The core of the proof lies in using a perturbation argument, that allows us to transfer the necessary properties from TFE model to IPM model for small enough $l>0$.
\smallskip
\begin{enumerate}
    \item[Step I.] \textit{Traveling wave solutions.} The main result states as follows.
\end{enumerate}
\begin{theorem}
\label{thm:TW-IPM}
Consider the two-tubes IPM equations~\eqref{eq:2tubes-diffusive-system-1}--\eqref{eq:peretok},~\eqref{eq:gravitational-Darcy-2tubes-c},~\eqref{eq:u1+u2=0}. 
\begin{itemize}
\item[A)] There exists  sufficiently small $l_0,\delta >0$ such that for all $v\in(v_1^*-\delta,v_1^*+\delta)$ and  $l\in(0,l_0)$ there exists a traveling wave solution $(c_{1},c_2,u_1,u_2,q)(y-vt)$ satisfying $(c_{1},c_2,u_1,u_2,q)(-\infty)=(-1,-1,0,0,0)$. It connects the states:
\begin{align*}
(-1,-1,0,0,0) \xrightarrow{TW} \left(c_1^*,c_2^*,\frac{c_2^*-c_1^*}2,\frac{c_1^*-c_2^*}2,0\right),
\end{align*}
where $c_1^*=c_1^*(v,l)$ and $c_2^*=c_2^*(v,l)$ depend continuously on $(v,l)$ at any point $(v,l)\in(v_1^*-\delta,v_1^*+\delta)\times(0,l_0)$ and satisfy as $l\to0$ for fixed~$v$
$$
c_{1}^*(v,l)\to -6v-1 \quad \text{ and }\quad c_{2}^*(v,l)\to -2v-1.
$$
Denote all possible pairs of states $(c_1^*(v,l),c_2^*(v,l))$ as $\mathcal{H}_{-1}(l)$, see Fig.~\ref{fig:HugoniotIPM}.

\item[B)] There exists  sufficiently small $l_0,\delta >0$ such that for all $v\in(v_2^*-\delta,v_2^*+\delta)$ and $l\in(0,l_0)$ there exists a traveling wave solution $(c_{1},c_2,u_1,u_2,q)(y-vt)$ 
satisfying  $(c_{1},c_2,u_1,u_2,q)(+\infty)=(1,1,0,0,0)$. It connects the states:
\begin{align*}
 \left(c_1^{**},c_2^{**},\frac{c_2^{**}-c_1^{**}}2,\frac{c_1^{**}-c_2^{**}}2,0\right)\xrightarrow{TW} (1,1,0,0,0),
\end{align*}
where $c_1^{**}=c_1^{**}(v,l)$ and $c_2^{**}=c_2^{**}(v,l)$ depend continuously on $(v,l)$ at any point $(v,l)\in(v_2^*-\delta,v_2^*+\delta)\times(0,l_0)$ and satisfy as $l\to0$ for fixed~$v$
$$
c_{1}^{**}(v,l)\to -2v+1 \quad \text{ and }\quad c_{2}^{**}(v,l)\to -6v+1.
$$
Denote all possible pairs of states $(c_1^{**}(v,l),c_2^{**}(v,l))$ as $\mathcal{H}_{1}(l)$, see Fig.~\ref{fig:HugoniotIPM}.

\end{itemize}
\end{theorem}

    To prove Theorem~\ref{thm:TW-IPM}  we put a traveling wave ansatz into the two-tubes IPM model and obtain a family of dynamical systems parametrized by the speed of the traveling wave $v\in\mathbb{R}$. We focus on finding the heteroclinic trajectories of these dynamical systems as they correspond to the traveling wave solutions. In particular, in Section~\ref{subsec:Peaceman-TW-main} we derive the traveling wave dynamical system for the two-tubes IPM. It is six-dimensional  and  has a slow-fast structure for $l>0$ sufficiently small. This allows to reduce the dimension of the dynamical system from six to four by using the geometric singular perturbation theory (GSPT). The formal  limiting system at $l=0$ coincides with the traveling wave dynamical system for the two-tubes TFE case for which we know the heteroclinic trajectories explicitly. Using transversality  argument obtained in Section \ref{subsec:transverse-TFE} we conclude that these orbits persist under small perturbations of the system. Theorem~\ref{thm:Peaceman-reformulated} states the key result on the existence of the heteroclinic trajectories and implies Theorem~\ref{thm:TW-IPM}.

\smallskip 
\begin{enumerate}
    \item[Step II.] \textit{Propagating terrace.}
\end{enumerate}

    By Step I, we established the existence of the traveling wave solutions satisfying $(c_1,c_2,u_1,u_2,q)(-\infty)=(-1,-1,0,0,0)$ or $(c_1,c_2,u_1,u_2,q)(+\infty)=(1,1,0,0,0)$.
    In order to prove Theorem~\ref{thm:main-Peaceman}, we show that the curves $\mathcal{H}_1(l)$ and $\mathcal{H}_{-1}(l)$ intersect. Indeed, Theorem~\ref{thm:TW-IPM} implies that the curves $\mathcal{H}_1(l)$ and $\mathcal{H}_{-1}(l)$ are small perturbations of the corresponding curves for TFE model in a small neighbourhood of the intersection point, see Fig.~\ref{fig:HugoniotIPM}. The intersection of $\mathcal{H}_1(l)$ and $\mathcal{H}_{-1}(l)$ is obtained from topological transversality of the intersection of $\mathcal{H}_1$ and $\mathcal{H}_{-1}$.

\section{Analysis of the two-tubes TFE model}
\label{sec:proof-thm2}
In this section we analyze the two-tubes TFE model~\eqref{eq:2tubes-diffusive-system-1}--\eqref{eq:peretok},~\eqref{eq:IPM-2-tubes-w},~\eqref{eq:gravitational-Darcy-TFE-2tubes}. First, in Section~\ref{sec:TW-TFE-main} we study traveling wave solutions and prove Theorem~\ref{thm:TW-TFE}. Second, in Section~\ref{subsec:prop-terrace-TFE} we prove Theorem~\ref{thm:main-TFE} on the existence and uniqueness of the propagating terrace. Third, in Section~\ref{subsec:transverse-TFE} we prove the auxiliary result that the heteroclinic orbits that correspond to traveling wave solutions can be represented as a transverse intersection of some stable and unstable manifolds.

\subsection{Traveling wave solutions}
\label{sec:TW-TFE-main}

\subsubsection{Traveling wave dynamical system. Derivation}
\label{sec:TW-TFE}

We are looking for a traveling wave solution $c_{1,2}(t,y)=\tilde{c}_{1,2}(y-vt)$
 of the system \eqref{eq:2tubes-diffusive-system-1}--\eqref{eq:peretok},~\eqref{eq:IPM-2-tubes-w},~\eqref{eq:gravitational-Darcy-TFE-2tubes}, connecting the states $(c_1^-, c_2^-)$ and $(c_1^+, c_2^+)$  
 (for simplicity we omit tildas)
\begin{align*}
 &c_{1,2}(t,y) = c_{1,2}(\xi),
& &\xi  = y - vt, & & c_{1,2}(\pm\infty) = c_{1,2}^\pm.
\end{align*}
In the above notation the system~\eqref{eq:2tubes-diffusive-system-1}--\eqref{eq:2tubes-diffusive-system-2} can be rewritten as
\begin{align}
\label{eq:dynsys-1}
    -v\partial_\xi c_1&=-\partial_\xi(u_1c_1)-f+\partial_{\xi\xi}c_1,
    \\
    \label{eq:dynsys-2}
    -v\partial_\xi c_2&=-\partial_\xi(u_2c_2)+f+\partial_{\xi\xi}c_2.
\end{align}
Let us rewrite these two second order ODEs as a system of four first order ODEs  using new unknown variables $(a,b,r,s)$:
\begin{equation}\label{eq:new-variables}
    a:=c_1-c_2;\quad b:=c_1+c_2; \quad r:=\partial_\xi(c_1-c_2); \quad s:=\partial_\xi (c_1+c_2).
\end{equation}
\begin{lemma}
\label{lm:TW-TFE}
The traveling wave dynamical system for the equations~\eqref{eq:2tubes-diffusive-system-1}--\eqref{eq:peretok}, \eqref{eq:IPM-2-tubes-w}, \eqref{eq:gravitational-Darcy-TFE-2tubes} takes the following form $\mathrm{(}$in variables~\eqref{eq:new-variables}$\mathrm{)}$:
\begin{align}
\label{eq:4dim-dyn-sys-TFE}
    \begin{split}
        \partial_{\xi}a&=r,
        \\
        \partial_\xi b&=s,
        \\
        \partial_\xi r&=-vr-\frac{sa}{2}+\frac{a|r|}{2},
        \\
        \partial_\xi s&=-vs-ra.
    \end{split}
\end{align}
\end{lemma}
\begin{remark}
    Note that right hand side of system \eqref{eq:4dim-dyn-sys-TFE} is locally Lipschitz and hence it satisfies existence and uniqueness conditions. At the same time it is not differentiable (due to the presence of $|r|$) and hence invariant manifold theory cannot be directly applied. The analysis of \eqref{eq:4dim-dyn-sys-TFE} below is similar to description of stable and unstable manifolds at the same time it is adapted to nonsmoothness of the system.
\end{remark}

\begin{proof}
Clearly, $\partial_\xi a=r$ and $\partial_\xi b=s$. Let's show how to obtain the equation on~$\partial_\xi s$. Summing up the equations~\eqref{eq:dynsys-1}--\eqref{eq:dynsys-2}, we get:
\begin{align*}
    -v\partial_\xi (c_1+c_2)&=-\partial_\xi(u_1c_1+u_2c_2)+\partial_{\xi\xi}(c_1+c_2).
\end{align*}
Using the expressions for $u_1$ and $u_2$, see relation~\eqref{eq:gravitational-Darcy-TFE-2tubes}, we obtain
\begin{align*}
    -v\partial_\xi (c_1+c_2)&=-\partial_\xi (c_2-c_1)\cdot (c_1-c_2)+\partial_{\xi\xi}(c_1+c_2).
\end{align*}
Substituting the new variables~\eqref{eq:new-variables}, we get the necessary equation on $\partial_\xi s$.

Now let's show how to obtain the equation on 
$\partial_\xi r$. Subtracting the equation~\eqref{eq:dynsys-2} from~\eqref{eq:dynsys-1}, we get:
\begin{align}
    -v\partial_\xi (c_1-c_2)&=-\partial_\xi(u_1c_1-u_2c_2)+\partial_{\xi\xi}(c_1-c_2)-2f,
    \nonumber
    \\
    \label{eq:4-eq}
    -v\partial_\xi (c_1-c_2)&=-\partial_\xi u_1\cdot (c_1+c_2)-u_1\cdot \partial_\xi (c_1+c_2)+\partial_{\xi\xi}(c_1-c_2)-2f.
\end{align}
The sum $(-\partial_\xi u_1\cdot (c_1+c_2)-2f)$ can be simplified:
\begin{align*}
    -\partial_\xi u_1\cdot (c_1+c_2)-2f&=
    \begin{cases}
        -\partial_\xi u_1 \cdot (c_2-c_1),& \partial_\xi u_1<0,\\
    -\partial_\xi u_1 \cdot (c_1-c_2),&\partial_\xi u_1\geq 0,
    \end{cases}
    =-\frac{a|r|}{2}.
\end{align*}
Substituting variables~\eqref{eq:new-variables} into equation~\eqref{eq:4-eq}, we get the necessary equation on~$\partial_\xi r$.
\end{proof}

\subsubsection{Traveling wave dynamical system. Analysis}
In this section we make preliminary analysis of the traveling wave dynamical system  and introduce invariant manifolds important for the proof of Theorem~\ref{thm:main-TFE}.

The set of fixed points for the system~\eqref{eq:4dim-dyn-sys-TFE} for any $v\in\mathbb{R}$  coincides with the set
\begin{align}
\label{eq:fixed-points-4-dim}
    M:=\{(a,b,r,s)\in\mathbb{R}^4: r=s=0\}.
\end{align}

Notice that the right hand side of the system~\eqref{eq:4dim-dyn-sys-TFE} does not depend on $b$, thus, we can first analyze the three-dimensional system:
\begin{align}
\label{eq:3dim-dyn-sys-TFE}
    \begin{split}
        \partial_{\xi}a&=r,
        \\
        \partial_\xi r&=-vr-\frac{sa}{2}+\frac{a|r|}{2},
        \\
        \partial_\xi s&=-vs-ra.
    \end{split}
\end{align}
By a straightforward computation one can check that for $r>0$ 
\begin{align}
\label{eq:id-1}
    \partial_\xi(r+s)& = \left(-v-\frac{a}{2}\right)(r+s);
    \\
    \label{eq:id-2}
    \partial_\xi(2r-s)& = (-v+a)(2r-s).
\end{align}
Similarly, for $r<0$
\begin{align*}
    \partial_\xi(r-s)& = \left(-v+\frac{a}{2}\right)(r-s);
    \\
    \partial_\xi(2r+s)& = (-v-a)(2r+s).
\end{align*}

As a consequence we obtain the following statement.
\begin{proposition}
    \label{rm:inv-manifolds}
The following two-dimensional manifolds are invariant for~\eqref{eq:3dim-dyn-sys-TFE} 
\begin{align*}
    I_1&=\{(a,r,s)\in\mathbb{R}^3:\;r>0\;\text{ and }\;s=2r\};
    \\
    I_2&=\{(a,r,s)\in\mathbb{R}^3:\;r>0\;\text{ and }\;s=-r\};
    \\
    I_3&=\{(a,r,s)\in\mathbb{R}^3:\;r<0\;\text{ and }\;s=r\};
    \\
    I_4&=\{(a,r,s)\in\mathbb{R}^3:\;r<0\;\text{ and }\;s=-2r\}.
\end{align*}
By uniqueness of solutions, the following regions are also invariant for \eqref{eq:3dim-dyn-sys-TFE}
\begin{align*}
    R_1&=\{(a,r,s)\in\mathbb{R}^3:s\geq -2r\;\text{ and }\;s\geq2r\};
    \\
    R_2&=\{(a,r,s)\in\mathbb{R}^3:\;s\leq 2r\;\text{ and }\;s\geq -r\};
    \\
    R_3&=\{(a,r,s)\in\mathbb{R}^3:\;s\leq-r\;\text{ and }\;s\leq r\};
    \\
    R_4&=\{(a,r,s)\in\mathbb{R}^3:\;s\geq r\;\text{ and }\;s\leq -2r\}.
\end{align*}
\end{proposition}
In what follows we will show that heteroclinic trajectories corresponding to traveling waves belong to the invariant manifolds $I_1$--$I_4$ (see Theorem~\ref{thm:TFE-reformulated} below).

\subsubsection{Traveling wave dynamical system. Heteroclinic trajectories}
\label{subsec:heteroclinic-orbits}
We are looking for all values of $a_0, b_0, v\in\mathbb{R}$ such that there exist a heteroclinic trajectory $(a(\xi), b(\xi),r(\xi),s(\xi))$ of the dynamical system~\eqref{eq:4dim-dyn-sys-TFE} that connects either a pair of critical points
\begin{align}
\label{eq:hetero-11}
    (a,b,r,s)(-\infty)=(a_0,b_0,0,0) \quad\text{and}\quad  (a,b,r,s)(+\infty)=(0,2,0,0),
\end{align}
 or points
\begin{align}
\label{eq:hetero--1-1}
    (a,b,r,s)(-\infty)=(0,-2,0,0) \quad\text{and}\quad  (a,b,r,s)(+\infty)=(a_0,b_0,0,0).
\end{align}

In order to prove Theorem~\ref{thm:TW-TFE}, we provide its reformulation in terms of heteroclinic trajectories for the traveling wave dynamical system~\eqref{eq:4dim-dyn-sys-TFE}. 

\begin{theorem}
    \label{thm:TFE-reformulated}
    The following statements are valid:
   
    {\bf Item 1.} For any $v>0$ there exist exactly two values of the pair $(a_0,b_0)$ such that there exists a heteroclinic trajectory of the system
    ~\eqref{eq:4dim-dyn-sys-TFE}, satisfying~\eqref{eq:hetero-11}.      Namely,
    \begin{itemize}
        \item $(a_0,b_0)=(+4v,-8v+2)$ and $r(\xi)<0$, $\xi\in\mathbb{R}$, on this trajectory;
        \item $(a_0,b_0)=(-4v,-8v+2)$ and $r(\xi)>0$, $\xi\in\mathbb{R}$, on this trajectory.
    \end{itemize}
    
    Also for any $v>0$ there does not exist any heteroclinic trajectory of the system~\eqref{eq:4dim-dyn-sys-TFE}, satisfying~\eqref{eq:hetero--1-1}. 

    {\bf Item 2.} For $v= 0$ the system \eqref{eq:4dim-dyn-sys-TFE} does not have a heteroclinic trajectory that satisfies 
    either $(a,b,r,s)(-\infty)=(0, b_0, 0, 0)$ or $(a,b,r,s)(+\infty)=(0, b_0, 0, 0)$ for any $b_0\in\mathbb{R}$.

    {\bf Item 3.} For any $v<0$ there exist exactly two values of the pair $(a_0,b_0)$ such that there exists a heteroclinic trajectory of the system~\eqref{eq:4dim-dyn-sys-TFE}, satisfying~\eqref{eq:hetero--1-1}. Namely,
    \begin{itemize}
        \item $(a_0,b_0)=(+4v,-8v-2)$ and $r(\xi)<0$, $\xi\in\mathbb{R}$, on this trajectory;
        \item $(a_0,b_0)=(-4v,-8v-2)$ and $r(\xi)>0$, $\xi\in\mathbb{R}$, on this trajectory.
    \end{itemize}
     Also for any $v<0$ there does not exist any heteroclinic trajectory of the system~\eqref{eq:4dim-dyn-sys-TFE}, satisfying~\eqref{eq:hetero-11}.

\end{theorem}
    Using formulas~\eqref{eq:new-variables}, $c_1=(a+b)/2$, $c_2=(b-a)/2$,  one immediately gets that Theorem~\ref{thm:TW-TFE} is a consequence of Theorem~\ref{thm:TFE-reformulated}. Let us prove Theorem~\ref{thm:TFE-reformulated}.

\begin{proof}
{\bf Item 1.}  Consider a trajectory $(a(\xi),b(\xi),r(\xi),s(\xi))$ of the system~\eqref{eq:4dim-dyn-sys-TFE}. Note that $(a(\xi),r(\xi),s(\xi))$ is a trajectory for the system~\eqref{eq:3dim-dyn-sys-TFE}.
 
We divide the proof into smaller lemmas.
\begin{lemma}
\label{lm:a0-not-0}
    Let $(a,r,s)(-\infty)=(a_0,0,0)$ for some $a_0\in\mathbb{R}$ and $(a(\xi),r(\xi),s(\xi))\not\equiv (a_0,0,0)$. Then $|a_0|\geq \frac45v$. In particular, there exists $\xi_0\in\mathbb{R}$ such that $a(\xi)$ is of constant sign for all $\xi<\xi_0$.
\end{lemma}
\begin{proof}
    Indeed, 
    \begin{align*}
        \partial_\xi(r^2+s^2)=2r\partial_\xi r+2s\partial_\xi s=2r\left(-vr-\frac{sa}{2}+\frac{a|r|}{2}\right)+2s(-vs-ra).
    \end{align*}
    Estimating $rs\leq(r^2+s^2)/2$ and $r|r|\leq (r^2+s^2)$, we get 
    \begin{align*}
        \partial_\xi(r^2+s^2)\leq(r^2+s^2)\left(-2v+\frac52|a|\right).
    \end{align*}
    By contradiction, if $|a_0|<\frac{4}5v$, then there exists $\xi_0\in\mathbb{R}$ such that   $\partial_\xi(r^2+s^2)\leq0$ for all $\xi<\xi_0$. Taking into account that $r^2+s^2\to0$ as $\xi\to-\infty$, we obtain $r^2+s^2\equiv0$ for all $\xi<\xi_0$, and by uniqueness of solutions this holds for all $\xi\in\mathbb{R}$.
\end{proof}

\begin{lemma}
\label{lm:sign-r}
There exists $\xi_1\in\mathbb{R}$ such that $r(\xi)$ is of constant sign for all $\xi<\xi_1$.
\end{lemma}
\begin{proof}
    Due to Proposition~\ref{rm:inv-manifolds}, the regions $R_i$, $i=1,2,3,4$, are invariant. For trajectory $(a,r,s)\in R_2$ or $(a,r,s)\in R_4$ the statement is obviously true.

    Consider $(a,r,s)\in R_1$. By Lemma~\ref{lm:a0-not-0}, for $\xi<\xi_0$ the sign of $a(\xi)$ is constant. Without loss of generality, assume that $a(\xi)>0$, $\xi<\xi_0$. Then
    \begin{align*}
        \partial_\xi r\biggr\rvert_{r=0}=-\frac{sa}{2}<0.
    \end{align*}
    This implies that there exists $\xi_1\in\mathbb{R}$, $\xi_1<\xi_0$ such that $r(\xi)<0$ for $\xi<\xi_1$. The case $(a,r,s)\in R_3$ is similar.
\end{proof}

\begin{lemma}
\label{lm:S1}
    Let function $\alpha: \mathbb{R}^3 \to \mathbb{R}$ satisfy
    \begin{equation}
    \label{eq1.0.5}
        \frac{d}{d\xi} \alpha(x(\xi)) = \alpha(x(\xi))\cdot\mu(x(\xi)),
    \end{equation}
    where $\mu:\mathbb{R}^3 \to \mathbb{R}$ is a continuous function. Assume that
\begin{equation}
\label{eq1.1}
   \lim\limits_{\xi \to -\infty} x(\xi) =x^-, \quad \mu(x^-)<0, \quad \alpha(x^-) = 0,
\end{equation}
then 
\begin{equation}\label{eq1.2}
\alpha(x(\xi)) \equiv 0, \quad \xi \in \mathbb{R}.
\end{equation}
\end{lemma}
\begin{proof}
    Let $\beta(\xi) = \alpha^2(x(\xi))$, then by \eqref{eq1.0.5}
\begin{equation}\label{eq1.3}
    \frac{d}{d\xi}\beta(\xi) = 2\beta(\xi)\cdot \mu(x(\xi)).
\end{equation}
    Due to \eqref{eq1.1} and continuity of $\mu$, there exists $\xi_0\in\mathbb{R}$ such that for all $\xi<\xi_0$ the inequality $\mu(x(\xi))<0$ holds. Note that $\beta(\xi) \geq 0$ and, hence, 
    $$
    \frac{d}{d\xi}\beta(\xi) \leq 0, \quad \xi<\xi_0.
    $$
    Thus, $\beta(\xi)$ is a nonincreasing nonnegative function on $(-\infty, \xi_0)$. Since $\beta(\xi) \to 0$ as $\xi \to -\infty$ we conclude that $\beta(\xi) = 0$ for $\xi< \xi_0$. Due to uniqueness of solution of \eqref{eq1.3}, $\beta(\xi) \equiv 0$ for all $\xi \in \mathbb{R}$, which implies \eqref{eq1.2}.
\end{proof}

\begin{lemma}
\label{lm:inv-man}
    Consider a trajectory $(a,r,s)$ of the system~\eqref{eq:3dim-dyn-sys-TFE} such that 
    \newline $(a,r,s)(-\infty)=(a_0,0,0)$, $a_0\in\mathbb{R}$. Then $(a,r,s)\in I_1\cup I_2\cup I_3\cup I_4$, where $I_k, k=1,2,3,4,$ are defined in Proposition~\ref{rm:inv-manifolds}.
\end{lemma}
\begin{proof}
    By Lemma~\ref{lm:sign-r}, there exists $\xi_1\in\mathbb{R}$ such that $r(\xi)$ is of constant sign for all $\xi<\xi_1$.
    
    Consider the case $r(\xi)>0$, $\xi<\xi_1$.  By Lemma~\ref{lm:a0-not-0}, we have $a(\xi)$ is of constant sign for $\xi<\xi_1$. 
    \begin{itemize}
        \item 
    If $a(\xi)>0$, $\xi<\xi_1$, then $-v-a/2<0$. Applying Lemma~\ref{lm:S1} to the relation~\eqref{eq:id-1} with $\alpha=r+s$, we obtain $r+s\equiv0$, that is $(a,r,s)\in I_2$.
    \item If $a(\xi)<0$, $\xi<\xi_1$, then $-v+a<0$. Applying Lemma~\ref{lm:S1} to the relation~\eqref{eq:id-2} with $\alpha=2r-s$, we obtain $2r-s\equiv0$, that is $(a,r,s)\in I_1$.
    \end{itemize}
    The case $r(\xi)<0$, $\xi<\xi_1$, is similar.    
\end{proof}

So we reduced the problem to finding heteroclinic trajectories in each of the invariant half-planes $I_k, k=1,2,3,4$.
    \begin{lemma}
    \label{lm:inside-inv-man}
    \begin{itemize}
        \item[(i)] 
        If $(a,r,s)\in I_1$, then there exist a unique $a_0\in\mathbb{R}$ such that there exists a heteroclinic trajectory, satisfying
        \begin{align}
        \label{eq:hetero-1}
        (a,r,s)(-\infty)=(a_0,0,0) \quad\text{and}\quad  (a,r,s)(+\infty)=(0,0,0).
    \end{align}
    Moreover, $a_0=-4v$.
        \item[(ii)] 
        If $(a,r,s)\in I_4$, then there exist a unique $a_0\in\mathbb{R}$ such that there exists a heteroclinic trajectory, satisfying~\eqref{eq:hetero-1}. 
        Moreover, $a_0=4v$.
        \item[(iii)]
        If $(a,r,s)\in I_2$ or $(a,r,s)\in I_3$, then there does not exist heteroclinic trajectories, satisfying~\eqref{eq:hetero-1}.
        \end{itemize}
    
    \end{lemma}
\begin{proof}

\textit{(i)} The two-dimensional system inside $I_1$ takes the form:
\begin{align*}
    \partial_\xi a&=r,
    \\
    \partial_\xi r&=r\left(-v-\frac{a}{2}\right).
\end{align*}
By a straightforward computation, we get
\begin{align*}
        \frac{d}{d\xi}\left( r+\left(v+\frac{a}{2}\right)^2\right)=0.
\end{align*}
This allows us to construct the phase portrait and write an explicit relation
\begin{align}
\label{eq:parabola}
         r=-\left(v+\frac{a}{2}\right)^2+r_0, \quad r_0\in\mathbb{R}.
\end{align}
Substituting it into the equation $\partial_\xi a=r$, it is clear, that there exists a unique~$r_0$ and $a_0$ such that the trajectory $(a,r,s)$ lies in the region $\{r\geq 0\}$ and  satisfies \eqref{eq:hetero-1}. Indeed, the condition 
 $(a,r,s)(+\infty)=(0,0,0)$ implies   $r_0=v^2$. The relation~\eqref{eq:parabola} immediately gives $a_0=-4v$.

\textit{(ii)} The proof is similar to \textit{(i)}. The two-dimensional system inside $I_4$ reads as:
\begin{align*}
    \partial_\xi a&=r,
    \\
    \partial_\xi r&=r\left(-v+\frac{a}{2}\right).
\end{align*}
By a straightforward computation, we get 
\begin{align}
\label{eq:parabola-1}
        r=\left(-v+\frac{a}{2}\right)^2+r_0, \quad r_0\in\mathbb{R}.
\end{align}
Substituting it into the equation $\partial_\xi a=r$, we obtain that there exists a unique~$r_0$ and $a_0$ such that the trajectory $(a,r,s)$ lies in the region $\{r\leq0\}$ and satisfies \eqref{eq:hetero-1}. Indeed, the condition 
 $(a,r,s)(+\infty)=(0,0,0)$ implies   $r_0=-v^2$.
 The relation~\eqref{eq:parabola-1} immediately gives $a_0=4v$.

\textit{(iii)} The analogous reasoning for the half-planes $I_2$ and $I_3$ implies
\begin{align*}
    \text{on }I_2:&\quad \frac{d}{d\xi}\left(r-\frac{(a-v)^2}{2}\right)=0;
    \qquad &&
    \text{on }I_3:\quad \frac{d}{d\xi}\left(r+\frac{(a+v)^2}{2}\right)=0.
\end{align*}
These relations allow to construct the phase portrait, and conclude that there are no heteroclinic trajectories in $I_2$ and $I_3$.  
\end{proof}

Let us finish the proof of Theorem~\ref{thm:TFE-reformulated} for $v>0$. By Lemma~\ref{lm:inv-man}, the heteroclinic trajectory $(a(\xi),r(\xi),s(\xi))$ necessarily lies in one of the four invariant half-planes $I_k, k=1,2,3,4$. By Lemma~\ref{lm:inside-inv-man}, there exist exactly two heteroclinic trajectories inside $I_k, k=1,2,3,4$ which satisfy~\eqref{eq:hetero-1}. 

In particular, if $(a,r,s)\in I_1$, then $s\equiv 2r$, thus, $\partial_\xi(b-2a)=0$. Taking into account~\eqref{eq:hetero-11}, we obtain
$b_0-2a_0=2$, which gives $b_0=-8v+2$.

If $(a,r,s)\in I_4$, then $s\equiv -2r$, thus, $\partial_\xi(b+2a)=0$. Taking into account~\eqref{eq:hetero-11}, we obtain
$b_0+2a_0=2$, which gives $b_0=-8v+2$.

This finishes the proof of Theorem~\ref{thm:TFE-reformulated} for $v>0$.

\medskip
{\bf Item 2.}    Assume that there exists a heteroclinic trajectory $(a(\xi), r(\xi), s(\xi))$ connecting $(a_0, 0, 0)$ and $(0, 0, 0)$. The first and the third equations of system  \eqref{eq:3dim-dyn-sys-TFE} imply that
    $$
    \frac{d}{d \xi}\left(s+\frac{a^2}2\right) = 0,
    $$
    and hence $s+\frac{a^2}2\equiv 0$. In particular, $a_0 = 0$ and the heteroclinic trajectory is actually homoclinic. Substituting $s = -a^2/2$ into \eqref{eq:3dim-dyn-sys-TFE}, we get the following system
    \begin{align}
\label{eq:2dim-dyn-sys-TFE-v0}
    \begin{split}
        \partial_{\xi}a&=r,
        \\
        \partial_\xi r&=\frac{a^3}{4}+\frac{a|r|}{2}.
    \end{split}
\end{align}
Analyzing the signs of $\partial_{\xi}a$ and $\partial_\xi r$ on lines $\{r = 0\}$ and $\{a=0\}$, we conclude that the sets 
$$
A_1 = \{r\geq 0, a \geq 0\}, \quad A_3 = \{r\leq 0, a\leq 0\}
$$
are positive invariant and sets 
$$
A_2 = \{r\leq 0, a \geq 0\}, \quad A_4 = \{r\geq 0, a\leq 0\}
$$    
are negative invariant, see Fig.~\ref{fig:v0}.
\begin{figure}[ht]
    \centering
    \includegraphics[width=0.3\textwidth]{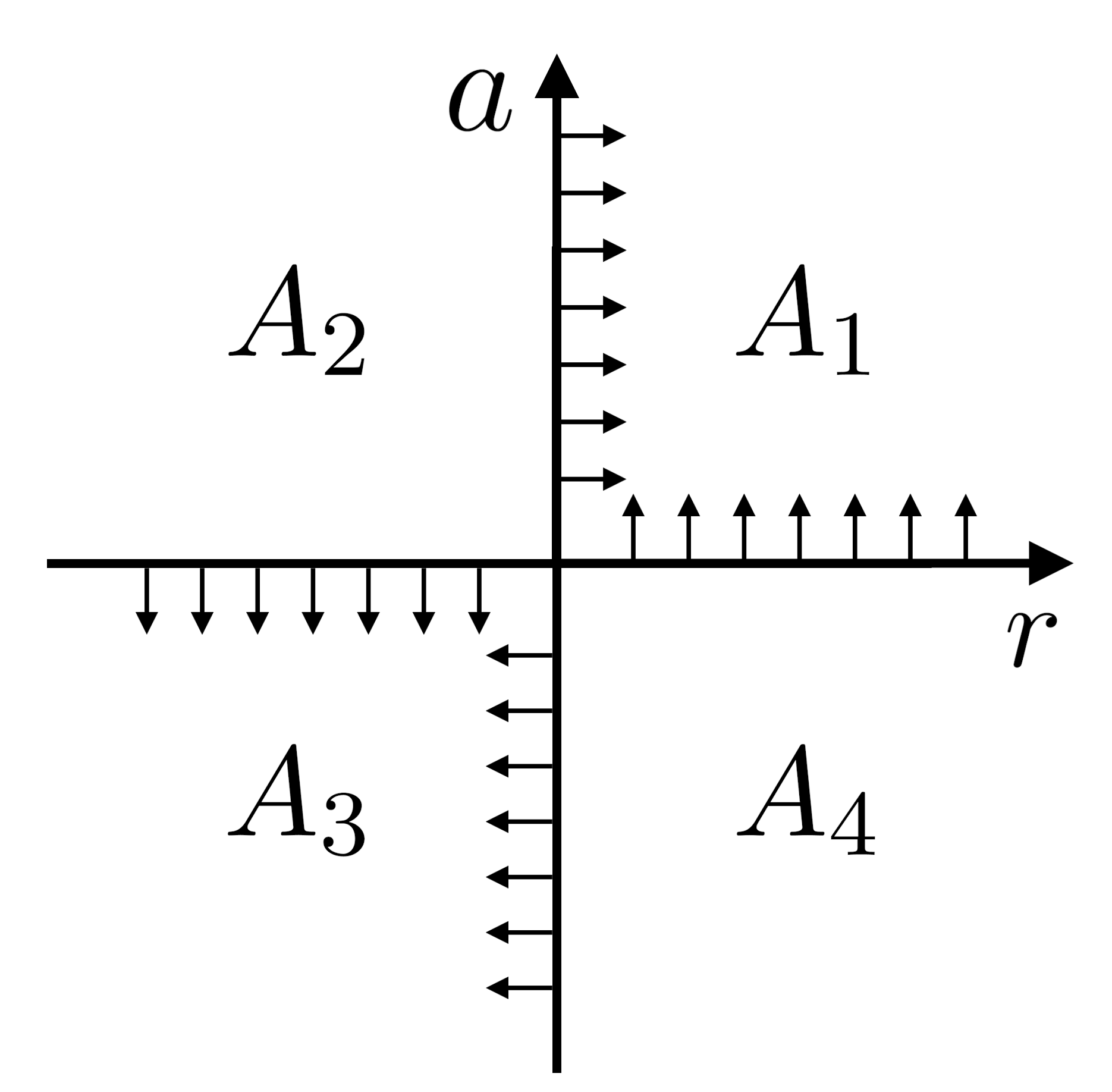}
    \caption{Vector field \eqref{eq:2dim-dyn-sys-TFE-v0} on axis $a=0$ and $r=0$.}
    \label{fig:v0}
\end{figure}
Note that in the sets $A_1$, $A_3$ the value of $a^2$ is monotonically increasing and in the sets $A_2$, $A_4$ the value of $r^2$ is monotonically decreasing. These  statements contradict to the existence of nontrivial homoclinic trajectory associated to the point $(a = 0, r = 0)$.

\medskip
{\bf Item 3.} The case $v<0$ can be managed similarly to $v>0$. We omit the~proof.
\end{proof}

\subsection{Existence of propagating terrace}
\label{subsec:prop-terrace-TFE}

 In this section we prove Theorem~\ref{thm:main-TFE}. By Theorem~\ref{thm:TW-TFE} we obtain:
\begin{itemize}
    \item for every $v<0$ there exist exactly one traveling wave solution $c_{1,2}$ connecting $\sigma_0=(-1,-1)$  with  each of the states
    \begin{align*}
\left[
\begin{array}{ll}
    \sigma_1^{*}(v)=(c_1^*(v),c_2^*(v))=(-2v-1,-6v-1) \\
    \sigma_1^{*}(v)=(c_1^*(v),c_2^*(v))=(-6v-1,-2v-1)
\end{array}
\right .
    \end{align*}
Note that $\mathcal{H}_{-1}:=\{\sigma_1^*(v): v<0\}$, see Fig.~\ref{fig:HugoniotTFE}.

    \item for every $v>0$ there exist exactly one traveling wave solution $c_{1,2}$ 
    connecting each of the states
    \begin{align*}
\left[
\begin{array}{ll}
    \sigma_1^{**}(v)=(c_1^{**}(v),c_2^{**}(v))=(-2v+1,-6v+1) \\
    \sigma_1^{**}(v)=(c_1^{**}(v),c_2^{**}(v))=(-6v+1,-2v+1)
\end{array}
\right .
    \end{align*}
    with $\sigma_2=(1,1)$. 
    Note that  $\mathcal{H}_1:=\{\sigma_1^{**}(v): v>0\}$, see Fig.~\ref{fig:HugoniotTFE}.
\end{itemize}
Notice that the intersection $\mathcal{H}_1\cap\mathcal{H}_{-1}$ consists of two points $(-\frac12,\frac12)$ and $(\frac12,-\frac12)$, and, thus, Theorem~\ref{thm:main-TFE} is proven.

\subsection{Transversal intersection of stable and unstable manifolds}
\label{subsec:transverse-TFE}
In this section we show that a heteroclinic trajectory, described in Theorem~\ref{thm:TFE-reformulated}, can be seen as a transversal intersection of suitable stable and unstable manifolds.  We will use this fact in  Section~\ref{sec:proof-peaceman} to prove the persistence of the heteroclinic trajectory under small perturbations.

We consider a system~\eqref{eq:4dim-dyn-sys-TFE} under the constraint $r\leq0$, that is 
\begin{align}
\label{eq:4dim-dyn-sys-TFE-r-negative}
    \begin{split}
        \partial_{\xi}a&=r,
        \\
        \partial_\xi b&=s,
        \\
        \partial_\xi r&=-vr-\frac{a}{2}(s+r),
        \\
        \partial_\xi s&=-vs-ra.
    \end{split}
\end{align}
In what follows we assume $(a,b,r,s)\in\mathbb{R}^4$. 
We recall that $M$, defined in~\eqref{eq:fixed-points-4-dim}, is the set of fixed points of~\eqref{eq:4dim-dyn-sys-TFE-r-negative}. By Theorem~\ref{thm:TFE-reformulated} Item~1, for $v>0$ there exists a heteroclinic trajectory $\Gamma$ between the fixed points: 
\begin{align*}
    A:=(a,b,r,s)(+\infty)&=(0,2,0,0),\quad\text{and}\quad
    \\
    \nonumber 
    B:=(a,b,r,s)(-\infty)&=
    (4v,-8v+2,0,0), 
\end{align*}
such that $r(\xi)<0$, $\xi\in\mathbb{R}$. 
 Fix a point $C \in \Gamma$. Denote by $\phi_\xi(X)$ an image of the point $X\in\mathbb{R}^4$ under the flow of the system~\eqref{eq:4dim-dyn-sys-TFE-r-negative} at  $\xi\in\mathbb{R}$.  For any point 
$X\in\mathbb{R}^4$ we denote by  $U_X$ some neighbourhood of $X$ in $\mathbb{R}^4$.  

\subsubsection{Stable manifold}
 First, notice that the tangent space at point $A$ has a splitting into stable and central parts:
\begin{align*}
\mathrm{T}_A\mathbb{R}^4=E^s(A)\oplus E^c(A),
\end{align*}
where the central part $E^c(A)$ is spanned by eigenvectors corresponding to zero eigenvalues $\lambda_{3,4}=0$:
\begin{align*}
    E^c(A)=\langle 
    \begin{pmatrix}
    1&0&0&0    
    \end{pmatrix}^T, 
    \begin{pmatrix}
    0&1&0&0    
    \end{pmatrix}^T
    \rangle,
\end{align*}
while the stable part $E^s(A)$ is spanned by eigenvectors corresponding to eigenvalues with negative real part $\lambda_{1,2}=-v$ (here $v>0$ by Theorem~\ref{thm:TFE-reformulated}, Item~1):
\begin{align}
\label{eq:EsA}
    E^s(A)=\langle 
    \begin{pmatrix}
    1 & 1 & -v & -v    
    \end{pmatrix}^T, 
    \begin{pmatrix}
    -1 & 2 & v & -2v    
    \end{pmatrix}^T
    \rangle.
\end{align}
By~\cite[Theorem~4.1(e)]{1970-HPS}, there exists a local stable manifold $W^s_{loc}(A)$ defined~by
\begin{align}\label{eq:Wsloc}
    W^s_{loc}(A)=\{X\in U_A: \phi_\xi(X)\to A\text{ as }\xi\to\infty\text{ and }\phi_\xi(X)\in U_A\text{ for all }\xi\geq0\},
\end{align}
which is tangent to $E^s(A)$ at point $A$. Notice that $\mathrm{dim} (W^s_{loc}(A))=2$.

Since $C \in \Gamma$ there exists $\xi_0<0$ and $x\in W^s_{loc}(A)$ such that $C=\phi_{\xi_0}(x)$. Denote
\begin{align} \label{eq:Wscomp}
  W^s_{\mathrm{comp}}(A):&= \{\phi_\xi(x)\in\mathbb{R}^4: \text{ for all }x\in W^s_{loc}(A), \xi\in[-\xi_0-1,0]\}.
\end{align}
Note that $C\in W^s_{\mathrm{comp}}(A)$.

In what follows we will need the following construction. Consider small enough $d>0$ and a set 
\begin{equation}\label{eq:defta}
    A^* = \{(a, b, 0, 0), \mbox{ where } |a|, |b-2| < d\},    
\end{equation}
    such that $\dim W^s_{loc}(x) = 2$ for any $x \in A^*$. 

\subsubsection{Unstable manifold}
The tangent space at point $B$ has a splitting into stable,  central and unstable parts:
\begin{align*}
\mathrm{T}_B\mathbb{R}^4=E^s(B)\oplus E^c(B)\oplus E^u(B),
\end{align*}
where the central part $E^c(B)$, analogously to the case for the point $A$, is spanned by eigenvectors corresponding to zero eigenvalues $\lambda_{3,4}=0$:
\begin{align*}
    E^c(B)=\langle 
    \begin{pmatrix}
    1&0&0&0    
    \end{pmatrix}^T, 
    \begin{pmatrix}
    0&1&0&0    
    \end{pmatrix}^T
    \rangle,
\end{align*}
while the stable part $E^s(B)$ is spanned by eigenvector corresponding to eigenvalue with negative real part $\lambda_{1}=-5v$:
\begin{align*}
    E^s(B)=\langle 
    \begin{pmatrix}
    1 & 1 & -5v & -5v    
    \end{pmatrix}^T
    \rangle,
\end{align*}
and the unstable part $E^u(B)$ is spanned by eigenvector corresponding to eigenvalue with positive real part $\lambda_{1}=v$:
\begin{align}\label{eq:13.1}
    E^u(B)=\langle 
    \begin{pmatrix}
    -1 & 1 & -v & 2v    
    \end{pmatrix}^T
    \rangle.
\end{align}

By~\cite[Theorem~4.1(be)]{1970-HPS}, there exists a unique  local unstable manifold $W^u_{loc}(B)$ defined by
\begin{align}
\label{eq:Wuloc}
    W^u_{loc}(B)=\{X\in U_B: \phi_\xi(X)\to B\text{ as }\xi\to-\infty\text{ and }\phi_\xi(X)\in U_B\text{ for all }\xi\leq0\},
\end{align}
which is tangent to $E^u(B)$ at point $B$. Notice that $\mathrm{dim} (W^u_{loc}(B))=1$. 

Let $\tilde{S}$~be the neighbourhood of the point $B$ inside $M$:
\begin{align}\label{eq:defts}
    \tilde{S}:&=\{(a,b,0,0): a\in(3.9v,4.1v) , b\in\bigl(0.9(-8v+2),1.1(-8v+2)\bigr) \}.
\end{align}
By~\cite[Theorem~4.1(ab)]{1970-HPS}, there exists a   unique  local unstable manifold $W^u_{loc}(\tilde{S})$ which is tangent to $E^u(X)\oplus E^c(X)$ for any $X\in \tilde{S}$. Notice that $\mathrm{dim}(W^u_{loc}(\tilde{S}))=3$.

Due to Proposition~\ref{rm:inv-manifolds}, we observe that the set 
\begin{align*}
    \{s+2r=0\}:=\{(a,b,r,s): s+2r=0\}
\end{align*}
is invariant and tangent to $E^u(X)\oplus E^c(X)$ for $X\in S$, thus $ W^u_{loc}(\tilde{S})\subset \{s+2r=0\}$.

Since $C \in \Gamma$ there exists $\xi_1\geq0$ and $y\in W^u_{loc}(\tilde{S})$ such that $C=\phi_{\xi_1}(y)$. Denote 
\begin{align}\label{eq:Wucomp}
  W^u_{\mathrm{comp}}(\tilde{S}):&=\{\phi_\xi(x)\in\mathbb{R}^4: \text{ for all }x\in W^u_{loc}(\tilde{S}), \xi\in[0,\xi_1+1]\}.
\end{align}
Note that $C\in W^u_{\mathrm{comp}}(\tilde{S})$.

\begin{figure}[ht]
    \centering
    \includegraphics[width=0.45\textwidth]{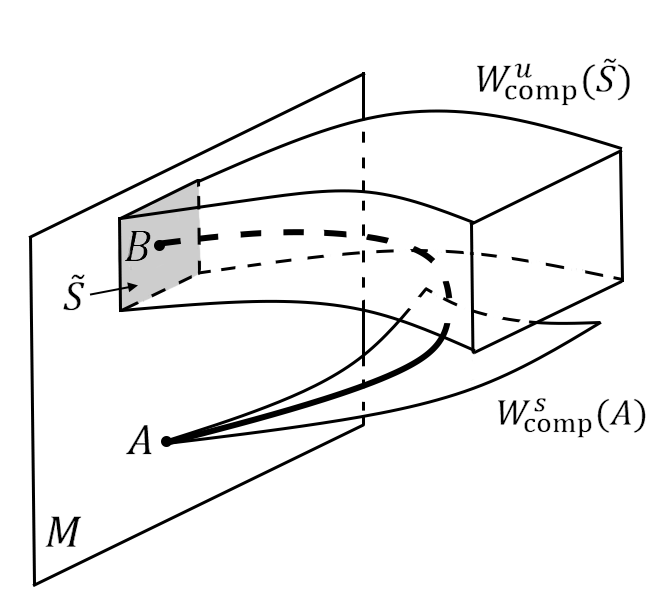}
    \caption{Schematic representation of the transversal intersection of the stable manifold $W^s_{\mathrm{comp}}(A)$ and unstable manifold $W^u_{\mathrm{comp}}(\tilde{S})$.}
    \label{fig:stable-unstable-transversal}
\end{figure}

\subsubsection{Transversal intersection}
The following Lemma claims that the manifolds $W^s_{\mathrm{comp}}(A)$ and $W^u_{\mathrm{comp}}(\tilde{S})$ intersect transversely in point $C$ (and, thus, in any point of $\Gamma$).

\begin{lemma}
The manifolds $W^s_{\mathrm{comp}}(A)$ and $W^u_{\mathrm{comp}}(\tilde{S})$  intersect transversely in point~$C$.
\end{lemma}
\begin{proof}
   Consider
    \begin{align*}
        W^u(\tilde{S}):=\{\phi_\xi(x)\in\mathbb{R}^4: \text{ for all }x\in W^u_{loc}(\tilde{S}), \xi\in[0,+\infty)\}.
    \end{align*}
    It is clear that $A\in \mathrm{Cl}(W^u(\tilde{S}))$ and $W^u(\tilde{S})\subset \{s+2r=0\}$. Note that
    \begin{align*}
        T_A(\{s+2r=0\})+T_A(W_{loc}^s(A))=\mathbb{R}^4.
    \end{align*}
    Hence, for any point $C^*=\phi_\xi(C)$ sufficiently close to $A$ (for some $\xi\in\mathbb{R}$), we obtain
    \begin{align*}    T_{C^*}(W^u(\tilde{S}))+T_{C^*}(W_{loc}^s(A))=\mathbb{R}^4.
    \end{align*}
    This implies that the intersection of the manifolds $W^s_{loc}(A)$ and $W^u(\tilde{S})$ is transversal in point $C^*$, hence the manifolds $W^s_{\mathrm{comp}}(A)$ and $W^u_{\mathrm{comp}}(\tilde{S})$ also intersect transversely in the point $C \in \Gamma$.
\end{proof}
\begin{remark}
    In Section~\ref{sec:proof-peaceman} we restrict ourselves to considering the systems~\eqref{eq:4dim-dyn-sys-TFE-r-negative} for $v$ sufficiently close to $v_1^*$. In this case, instead of considering the sets $\tilde{S}$ depending on $v$ (see formula~\eqref{eq:defts}), we define the unique set $S^*$ for all $v\in(v_1^*-\delta,v_1^*+\delta)$ for some small enough $\delta>0$:
    \begin{align}\label{eq:S*}
        S^*:=\{(a,b,0,0): a\in(3.9v^*_1,4.1v^*_1) , b\in\bigl(0.9(-8v^*_1+2),1.1(-8v^*_1+2)\bigr) \}.
    \end{align}
    By the same argument as before, the manifolds $W^s_{\mathrm{comp}}(A)$ and $W^u_{\mathrm{comp}}(S^*)$ intersect transversally. 
\end{remark}

\section{Analysis of the two-tubes IPM model}
\label{sec:proof-peaceman}

In this section we analyze the two-tubes IPM model~\eqref{eq:2tubes-diffusive-system-1}--\eqref{eq:peretok},~\eqref{eq:gravitational-Darcy-2tubes-c},~\eqref{eq:u1+u2=0}. First, in Section~\ref{subsec:Peaceman-TW-main} we study traveling wave solutions and prove Theorem~\ref{thm:TW-IPM}. Second, in Section~\ref{subsec:prop-terrace-IPM} we prove Theorem~\ref{thm:main-Peaceman} on the existence of the propagating terrace.

\subsection{Traveling wave solutions}
\label{subsec:Peaceman-TW-main}
\subsubsection{Traveling wave dynamical system. Derivation}
\label{subsec:Peaceman-TW}
We are looking for a traveling wave solution of the system \eqref{eq:2tubes-diffusive-system-1}--\eqref{eq:peretok},~\eqref{eq:gravitational-Darcy-2tubes-c},~\eqref{eq:u1+u2=0}
(for simplicity we denote the new unknown functions of one variable by the same notation)
\begin{align*}
(c_1,c_2, u_1, u_2,q)(t,y)=(c_1,c_2, u_1, u_2,q)(\xi),\qquad \xi=y-vt,
 \end{align*}
connecting the states $(c_1^-, c_2^-,u_1^-,u_2^-,q^-)$ and $(c_1^+, c_2^+,u_1^+,u_2^+,q^+)$, i.e. 
\begin{align*}
(c_1,c_2, u_1, u_2,q)(\pm\infty)=(c_1^\pm,c_2^\pm, u_1^\pm, u_2^\pm,q^\pm).
\end{align*}

In variable $\xi$ the system \eqref{eq:2tubes-diffusive-system-1}--\eqref{eq:peretok},~\eqref{eq:gravitational-Darcy-2tubes-c},~\eqref{eq:u1+u2=0} consists of the two equations \eqref{eq:dynsys-1}--\eqref{eq:dynsys-2} as before, and the additional ones:
\begin{align}
\label{eq:dynsys-3}
    \partial_\xi q& = u_1 - u_2 + c_1 - c_2,&
    \partial_\xi u_1 &= q/l^2,&
    \partial_\xi u_2 &= -q/l^2.
\end{align}

Using the unknown variables $(a,b,r,s)$, defined in formula~\eqref{eq:new-variables}, we write the system~\eqref{eq:dynsys-1}--\eqref{eq:dynsys-2} and~\eqref{eq:dynsys-3} in the form of a traveling wave dynamical system~\eqref{eq:6dim-dyn-sys-Peaceman}, as is stated in the following Lemma.

\begin{lemma}
The traveling wave dynamical system for the  two-tubes IPM equations~\eqref{eq:2tubes-diffusive-system-1}--\eqref{eq:peretok},~\eqref{eq:gravitational-Darcy-2tubes-c},~\eqref{eq:u1+u2=0} takes the following form $\mathrm{(}$in variables~\eqref{eq:new-variables}$\mathrm{)}$:
\begin{align}
\label{eq:6dim-dyn-sys-Peaceman}
    \begin{split}
        \partial_{\xi}a&=r,
        \\
        \partial_\xi b&=s,
        \\
        \partial_\xi r&=-vr+u_1s+a\cdot \frac{|q|}{l^2},
        \\
        \partial_\xi s&=-vs+u_1r+a\cdot \frac{q}{l^2},
        \\
        \partial_\xi u_1&=\frac{q}{l^2},
        \\
        \partial_\xi q&=2u_1+a.
    \end{split}
\end{align}
\end{lemma}
\begin{proof}
As before it is clear that  $\partial_\xi a=r$ and $\partial_\xi b=s$. Let's show how to obtain the equation on 
$\partial_\xi s$. Summing up the equations~\eqref{eq:dynsys-1}--\eqref{eq:dynsys-2}, we get:
\begin{align*}
    -v\partial_\xi (c_1+c_2)&=-\partial_\xi(u_1c_1+u_2c_2)+\partial_{\xi\xi}(c_1+c_2),
    \\
    -v\partial_\xi (c_1+c_2)&=-\partial_\xi u_1\cdot (c_1-c_2)- u_1\cdot \partial_\xi(c_1-c_2)+\partial_{\xi\xi}(c_1+c_2).
\end{align*}
Using the relations~\eqref{eq:gravitational-Darcy-2tubes-c} and~\eqref{eq:new-variables}, we get the necessary equation on $\partial_\xi s$.

Now let's show how to obtain the equation on 
$\partial_\xi r$. Subtracting the equation~\eqref{eq:dynsys-2} from~\eqref{eq:dynsys-1}, we get:
\begin{align*}
    -v\partial_\xi (c_1-c_2)&=-\partial_\xi(u_1c_1-u_2c_2)+\partial_{\xi\xi}(c_1-c_2)-2f,
    \\
    -v\partial_\xi (c_1-c_2)&=-\partial_\xi u_1\cdot (c_1+c_2)-u_1\cdot \partial_\xi (c_1+c_2)+\partial_{\xi\xi}(c_1-c_2)-2f.
\end{align*}
The sum $(-\partial_\xi u_1\cdot (c_1+c_2)-2f)$ can be simplified:
\begin{align*}
    -\partial_\xi u_1\cdot (c_1+c_2)-2f&=
    \begin{cases}
        -\partial_\xi u_1 \cdot (c_2-c_1),& \partial_\xi u_1<0,\\
    -\partial_\xi u_1 \cdot (c_1-c_2),&\partial_\xi u_1\geq 0,
    \end{cases}
    =
    \begin{cases}
        \frac{aq}{l^2},& q<0,\\
        -\frac{aq}{l^2},&q\geq 0,
    \end{cases}
    =-\frac{a|q|}{l^2}.
\end{align*}
Substituting the new variables~\eqref{eq:new-variables} into the equation~\eqref{eq:4-eq}, we get the necessary equation on $\partial_\xi r$.
\end{proof}

Assume $q\geq0$. We rewrite the system~\eqref{eq:6dim-dyn-sys-Peaceman} using the following change of variables: 
    \begin{align*}
    r_1:=r-u_1a-\frac{a^2}{2}, 
    \qquad 
    s_1:=s-u_1a-\frac{a^2}{2},
    \qquad 
    q_1:=\frac{q}{l}.
    \end{align*}
    In new variables the system~\eqref{eq:6dim-dyn-sys-Peaceman} becomes a slow-fast system in the standard form (as $l\to0$):
\begin{align}
\label{eq:6dim-dyn-sys-Peaceman-ver2}
    \begin{split}
        \partial_{\xi}a&=r_1+u_1a+\frac{a^2}{2},
        \\
        \partial_\xi b&=s_1+u_1a+\frac{a^2}{2},
        \\
        \partial_\xi r_1&=-v\left(r_1+u_1a+\frac{a^2}{2}\right)+u_1(s_1-r_1)-a\left(r_1+u_1a+\frac{a^2}{2}\right),
        \\
        \partial_\xi s_1&=-v\left(s_1+u_1a+\frac{a^2}{2}\right)-a\left(r_1+u_1a+\frac{a^2}{2}\right),
        \\
        l\cdot \partial_\xi u_1&=q_1,
        \\
        l \cdot \partial_\xi q_1&=2u_1+a.
    \end{split}
\end{align}
If we formally consider $l=0$ in~\eqref{eq:6dim-dyn-sys-Peaceman-ver2}, we get $u_1=-a/2=(c_2-c_1)/2$ and this expression coincides with the velocities in TFE model~\eqref{eq:gravitational-Darcy-TFE-2tubes}. Substituting this relation into the first four equations of the system~\eqref{eq:6dim-dyn-sys-Peaceman-ver2}, we obtain the traveling wave dynamical system for the TFE model~\eqref{eq:4dim-dyn-sys-TFE} (note that formally we have $r_1=r$ and $s_1=s$) under the restriction $r\leq0$. 

Notice that the same reasoning for $q\leq0$, leads to the traveling wave dynamical system for the TFE model~\eqref{eq:4dim-dyn-sys-TFE} under the restriction $r\geq0$. 

This puts us into the framework of geometric singular perturbation theory (e.g. see the original work of Fenichel~\cite{fenichel-1979}), and determines our general strategy: knowing the existence of heteroclinic trajectory for the ``unperturbed'' system for $l=0$ (two-tubes TFE equations), we prove that the heteroclinic trajectory persists under perturbation for small enough $l>0$ (two-tubes IPM equations).

The set of fixed points for the system~\eqref{eq:6dim-dyn-sys-Peaceman-ver2} for any $v\in\mathbb{R}$ and $l>0$  is
\begin{align*}
    M:=\bigl\{(a,b,r_1,s_1,u_1,q_1)\in\mathbb{R}^6: r_1=s_1=q_1=0, u_1=-\frac{a}{2}\bigr\}.
\end{align*}

\subsubsection{Traveling wave dynamical system. Heteroclinic trajectories}
\label{subsec:thm4-proof}

We are looking for all values of $a_0, b_0, v\in\mathbb{R}$ such that there exist a heteroclinic trajectory $(a(\xi), b(\xi),r_1(\xi),s_1(\xi),u_1(\xi),q_1(\xi))$ of the dynamical system~\eqref{eq:6dim-dyn-sys-Peaceman-ver2} that connects either a pair of critical points
\begin{align}
\begin{split}
    \label{eq:hetero-11-peaceman}
    (a,b,r_1,s_1,u_1,q_1)(-\infty)&=(a_0,b_0,0,0,-\frac{a_0}2,0) \\  (a,b,r_1,s_1,u_1,q_1)(+\infty)&=(0,2,0,0,0,0),
\end{split}
\end{align}
 or points
\begin{align}
\begin{split}
    \label{eq:hetero--1-1-peaceman}
    (a,b,r_1,s_1,u_1,q_1)(-\infty)&=(0,-2,0,0,0,0) \\
    (a,b,r_1,s_1,u_1,q_1)(+\infty)&=(a_0,b_0,0,0,-\frac{a_0}2,0).
\end{split}
\end{align}

In order to prove Theorem~\ref{thm:TW-IPM}, we provide its reformulation in terms of heteroclinic trajectories for the traveling wave dynamical system~\eqref{eq:6dim-dyn-sys-Peaceman-ver2}. 

\begin{theorem}
    \label{thm:Peaceman-reformulated}
    The following statements are valid:

    {\bf Item 1.} There exist $l_0, \delta>0$ such that for any $v\in(v_1^*-\delta,v_1^*+\delta)$ and for any $l\in[0,l_0)$ there exists a pair $(a_0,b_0)=(a_0(v,l),b_0(v,l))$ such that there exists a heteroclinic trajectory $\Gamma_{v, l}$ of the system~\eqref{eq:6dim-dyn-sys-Peaceman-ver2}, satisfying~\eqref{eq:hetero-11-peaceman}. Moreover, we have:
    \begin{itemize}
        \item[(a)] Continuous dependence:  $a_0$ and $b_0$ depend continuously on $(v,l)$ at any point $(v,l)\in(v_1^*-\delta,v_1^*+\delta)\times[0,l_0)$.
        \item[(b)] Consistency: $q_1(\xi)\geq0$, $\xi\in\mathbb{R}$, on $\Gamma_{v, l}$.
    \end{itemize}
    
    {\bf Item 2.} There exist $l_0, \delta>0$ such that for any $v\in(v_2^*-\delta,v_2^*+\delta)$ and for any  $l\in[0,l_0)$ there exists a pair $(a_0,b_0)=(a_0(v,l),b_0(v,l))$ such that there exists a heteroclinic trajectory $\Gamma'_{v, l}$ of the system~\eqref{eq:6dim-dyn-sys-Peaceman-ver2}, satisfying~\eqref{eq:hetero--1-1-peaceman}. Moreover, we have:
     \begin{itemize}
        \item[(a)]  Continuous dependence:  $a_0$ and $b_0$ depend continuously on $(v,l)$ at any point $(v,l)\in(v_2^*-\delta,v_2^*+\delta)\times[0,l_0)$. 
        \item[(b)] Consistency: $q_1(\xi)\geq0$, $\xi\in\mathbb{R}$, on $\Gamma'_{v, l}$.
    \end{itemize}
\end{theorem}

\begin{remark}
\label{rmk:limits}
    Note that for $l=0$ the system~\eqref{eq:6dim-dyn-sys-Peaceman-ver2} coincides with the system~\eqref{eq:4dim-dyn-sys-TFE}. This allows us to formally  define $a_0(v,0)$ and $b_0(v,0)$ from Theorem~\ref{thm:TFE-reformulated}. In particular,  Theorem~\ref{thm:Peaceman-reformulated}, Item~1, claims the continuous dependence of $a_0$ and $b_0$ at the point $(v,0)$,  $v\in (v_1^*-\delta,v_1^*+\delta)$. Thus, as $l\to0$ we have:
    $$
    \mbox{$q_1\geq0$: $a_0(v,l)\to 4v$ and $b_0(v,l)\to -8v-2$ by Theorem~\ref{thm:TFE-reformulated}, Item 3, case $r(\xi)<0$.}
    $$
    
    Analogously, Theorem~\ref{thm:Peaceman-reformulated}, Item~2, claims the continuous dependence of $a_0$ and $b_0$ at the point $(v,0)$,  $v\in (v_2^*-\delta,v_2^*+\delta)$. Thus, as $l\to0$ we have:
    $$
    \mbox{$q_1\geq0$: $a_0(v,l)\to 4v$ and $b_0(v,l)\to -8v+2$ by Theorem~\ref{thm:TFE-reformulated}, Item 1, case $r(\xi)<0$.}
    $$   
\end{remark}

\begin{remark}
\label{rmk:uniqueness}
We do not claim uniqueness of heteroclinic solutions of \eqref{eq:6dim-dyn-sys-Peaceman} due to possibility of trajectories with not constant sign of $q$, which is not covered by our techniques. However, we expect that corresponding heteroclinic of \eqref{eq:6dim-dyn-sys-Peaceman} is unique.
\end{remark}

    The analogue of Theorem~\ref{thm:Peaceman-reformulated} is also valid for the analogue of the system~\eqref{eq:6dim-dyn-sys-Peaceman-ver2} if we assume $q_1\leq0$ instead of $q_1\geq0$.

    Using formulas~\eqref{eq:new-variables}, $c_1=(a+b)/2$, $c_2=(b-a)/2$,~\eqref{eq:hetero-11-peaceman},~\eqref{eq:hetero--1-1-peaceman} and Remark~\ref{rmk:limits}  one immediately gets Theorem~\ref{thm:TW-IPM} as a consequence of Theorem~\ref{thm:Peaceman-reformulated}. Let us prove Theorem~\ref{thm:Peaceman-reformulated}.

\begin{proof}
    We prove Theorem~\ref{thm:Peaceman-reformulated}, Item 1, in four steps.

\textbf{Step 1:} We apply the geometric singular perturbation theory to the system~\eqref{eq:6dim-dyn-sys-Peaceman-ver2}, and reduce the problem to studying the heteroclinic trajectories in a 4-dimensional dynamical system~\eqref{eq:4-dim-F-ve}, described in Lemma~\ref{lem:He}. Notice that this system~\eqref{eq:4-dim-F-ve} for $l>0$ is a small perturbation of the dynamical system for TFE model ($l=0$), studied in detail in Section~\ref{sec:TW-TFE-main}.

\textbf{Step 2:} Since the heteroclinic trajectory of the unperturbed system ($l=0$) is a transverse intersection of some suitable stable and unstable manifolds (see Section~\ref{subsec:transverse-TFE}), then it persists under small perturbation for $l>0$. This guarantees the existence of a heteroclinic trajectory for~\eqref{eq:4-dim-F-ve}, and thus, for~\eqref{eq:6dim-dyn-sys-Peaceman-ver2} too.

\textbf{Step 3:} We prove the continuous dependence of the limiting points of the heteroclinic trajectory on the parameters $v$ and $l$ (Items 1a, 2a of the Theorem~\ref{thm:Peaceman-reformulated}).

\textbf{Step 4:} We prove that the found heteroclinic trajectory, indeed, satisfies the original assumption $q_1\geq0$ (Items 1b, 2b of the Theorem~\ref{thm:Peaceman-reformulated}).

The proof of Theorem~\ref{thm:Peaceman-reformulated}, Item 2, is analogous to Item 1, so we omit it.

\subsubsection*{Step 1: application of geometric singular perturbation theory}
In this section we apply the geometric singular perturbation theory (GSPT) for the system~\eqref{eq:6dim-dyn-sys-Peaceman-ver2}. This step allows to reduce the dimension of the dynamical system under consideration (from 6-dimensional to 4-dimensional). The original ideas on GSPT are due to N. Fenichel~\cite{fenichel-1979}, the more recent advances on the theory can be found in books~\cite{2015-Kuehn} and~\cite{2020-Wechselberger} (and references therein).

The system~\eqref{eq:6dim-dyn-sys-Peaceman-ver2}  has a structure of a ``slow-fast'' system (we denote $\partial_\xi$ by~$\dot{\phantom{a}}$):
\begin{align*}
\begin{cases}
    \quad\dot{X}&= F(X,Y),\\
    l\cdot \dot{Y}&=AY - G(X),
\end{cases}
\end{align*}
where 
$X=\begin{pmatrix}
    a&b&r_1&s_1
\end{pmatrix}^T\in\mathbb{R}^4$ and $Y=\begin{pmatrix}
    u_1 &q_1
\end{pmatrix}^T\in\mathbb{R}^2$, 
$A=\begin{pmatrix}
    0&1
    \\
    2&0
\end{pmatrix}$; $F(X,Y)$ and $G(X)$ are the concrete vector functions. While in system~\eqref{eq:6dim-dyn-sys-Peaceman-ver2} variable $v$ is a parameter, in this subsection it is convenient to include $v$ as a constant variable. Thus, we consider an extended system:
\begin{align}
\label{eq:slow-system}
\begin{cases}
    \quad\dot{v}&= 0,\\
    \quad\dot{X}&= F(X,Y, v),\\
    l\cdot \dot{Y}&=AY - G(X).
\end{cases}
\end{align}
System~\eqref{eq:slow-system} is usually called ``slow'' system. For any $l>0$ we obtain an equivalent ``fast'' system through the change of variables $\eta=\xi/ l$ (we denote $\partial_\eta$ by~$'$):
\begin{align}
\label{eq:fast-system}
\begin{cases}
     v'&= 0,
     \\
     X'&= l\cdot F(X,Y, v),
     \\
    Y'&=\quad AY - G(X).
\end{cases}
\end{align}
The set of fixed points of the fast system~\eqref{eq:fast-system} for $l=0$ is called a \textit{critical manifold}. In our case for the fast-slow system~\eqref{eq:slow-system}--\eqref{eq:fast-system} the critical manifold is:
\begin{align*}
    K=\{(v, X,Y)\in\mathbb{R}^7:AY=G(X)\},\qquad \mathrm{codim} (K)=2.
\end{align*}
    The critical  manifold $K$ is \textit{normally hyperbolic} since the matrix of the first derivatives $D_Y(AY-G(X))\rvert_K=A$ has two eigenvalues $\lambda=\pm\sqrt{2}$ with non-zero real part.

Consider a compact submanifold $K_0\subset K$. Due to normal hyperbolicity of $K$, we can apply Fenichel's result~\cite[Theorem 9.1]{fenichel-1979}. Thus, there exists a manifold $K_l \subset \mathbb{R}^7$ which is normally hyperbolic locally invariant under the flow of the system~\eqref{eq:slow-system}, diffeomorphic to~$K_0$, and for some  $H \in C^1$ the following inclusion is valid
    \begin{align*}
        K_l\subset\{(v,X,Y)\subset\mathbb{R}^7: AY=G(X)+l\cdot H(X, v, l)\}.
    \end{align*}
As we will see below, in order to prove the existence of the heteroclinic trajectory, it is enough to restrict ourselves to the analysis of the dynamics inside the invariant manifold~$K_l$.

\begin{lemma}\label{lem:He}
    Let $(v(\xi), X(\xi), Y(\xi)) \subset K_l$ be a trajectory of \eqref{eq:slow-system}, then there exists a function $F_{v, l}:\mathbb{R}^4 \to \mathbb{R}^4$ such that
    \begin{enumerate}
        \item $v(\xi) \equiv v^*$ for some $v^*\in\mathbb{R}$, and 
        \begin{align}
        \label{eq:4-dim-F-ve}
            \dot{X} = F_{v^*, l}(X).
        \end{align}
        \item $F_{v,0}$ coincides with the right-hand side of \eqref{eq:4dim-dyn-sys-TFE-r-negative}.
    \end{enumerate}
    Moreover, for $v$ close to $v_1^*$ and small enough $l>0$ the following holds:
    \begin{enumerate}\setcounter{enumi}{2}
        \item $F_{v, l}(X) = 0$ for $X \in A^* \cup S^*$, where $A^*$ is defined by \eqref{eq:defta}, and $S^*$ is defined by \eqref{eq:S*}.
        \item $F_{v,l}$ continuously depends on $(v,l)$ in the $C^1_{loc}$-topology.
    \end{enumerate}
\end{lemma}
\begin{proof} 
    Equality $v(\xi) \equiv v^*$ for some $v^*\in\mathbb{R}$ is a straightforward consequence of~\eqref{eq:slow-system}. Since $K_l$ is invariant and determined by relation $ AY=G(X)+l\cdot H(X,v^*, l)$, we define 
    $$
    F_{v^*,l}(X) := F(X, A^{-1}(G(X) + l\cdot  H(X,v^*,l)), v^*),  
    $$
which guarantees Item 1. Note that $F_{v,0}(X) = F(X, A^{-1}G(X), v)$ coincides with the right-hand side of \eqref{eq:4dim-dyn-sys-TFE-r-negative}, thus Item~2 is proved. Differentiability of $H$, $G$ and $F$   implies Item~4.

 Consider a point $X = (a, b, 0, 0) \in A^* \cup S^*$, then point $(a, b, 0, 0, -a/2, 0)$ is a fixed point of \eqref{eq:6dim-dyn-sys-Peaceman-ver2} for any $v, l$. Note that for small $l>0$ the point $(v, a, b, 0, 0, -a/2, 0)$ lies in a small neighborhood of $K_l$. Since $K_l$ is invariant and normally hyperbolic, any trajectory staying in a small neighborhood of it actually belongs to it. Hence, $(v, a, b, 0, 0, -a/2, 0) \in K_l$, which concludes the proof of Item~3.
\end{proof}

\subsubsection*{Step 2: proof of the intersection of stable and unstable manifolds for perturbed system} The constructions in this subsection are similar to \cite[Theorem~3.1]{SZMOLYAN1991}.
Consider a small disc $\mathfrak{D}\in\mathbb{R}^4$ transversal to the vector field $F_{v_1^*,0}$ at point $C \in W^u_{\mathrm{comp}}(S^*) \cap W^s_{\mathrm{comp}}(A)$. Denote by $\phi_\xi^{v^*,l}(X)$ an image of the point $X\in\mathbb{R}^4$ under the flow of the system~\eqref{eq:4-dim-F-ve}  at $\xi\in\mathbb{R}$, see Lemma~\ref{lem:He}. 

\begin{lemma}
\label{lem:inte}
    For small enough $l>0$ and $v$ close enough to $v_1^*$ there exists invariant manifolds $W^u_{loc, v, l}(S^*)$, $W^u_{\mathrm{comp}, v,l}(S^*)$, $W^s_{loc, v, l}(A)$, $W^s_{\mathrm{comp}, v, l}(A)$ of the flow $F_{v, l}$ defined similarly to \eqref{eq:Wuloc}, \eqref{eq:Wucomp}, \eqref{eq:Wsloc}, \eqref{eq:Wscomp}. There exists point of intersection   
    \begin{equation*}
    C(v, l) = W^u_{\mathrm{comp}, v, l}(S^*) \cap W^s_{\mathrm{comp}, v, l}(A) \cap \mathfrak{D},
    \end{equation*}
    which depends continuously on $v, l$ for $(v,l)\in(v_1^*-\delta,v_1^*+\delta)\times[0,l_0)$ for some $\delta, l_0>0$.
\end{lemma}
\begin{proof}
    The lemma is a straightforward consequence of the persistence of stable and unstable manifolds for normally hyperbolic sets \cite[Theorem 4.1]{1970-HPS}, Lemma \ref{lem:He} and transversality of the manifolds $W^u_{\mathrm{comp}}(S^*)$ and $W^s_{\mathrm{comp}}(A)$. Below we provide a more detailed exposition.
  
    Note that the set $S^*$ is normally hyperbolic for the flow $F_{v_1^*, 0}$. According to \cite[Theorem 4.1(f)]{1970-HPS} for the flow $F_{v, l}$ with $(v, l)$ close to $(v_1^*, 0)$ there exists a unique locally invariant set $S^*_{v,l}$ in a neighborhood of $S^*$  and its local unstable manifold $W^u_{loc, v,l}(S^*_{v,l})$ continuously depending on $v, l$ in the $C^1$-topology. Due to Lemma~\ref{lem:He}, Item~3, the set $S^*$ is invariant for $F_{v,l}$, hence  $ S^* = S^*_{v,l}$ and $W^u_{loc, v,l}(S^*_{v,l}) = W^u_{loc, v, l}(S^*)$. Note that $W^u_{loc, v_1^*, 0}(S^*) = W^u_{loc}(S^*)$.

    Similarly, the set $A^*$ is normally hyperbolic (attracting) for $F_{v_1^*, 0}$ and according to \cite[Theorem 4.1(f)]{1970-HPS} and Lemma \ref{lem:He}, Item 3, the stable manifold of $W^s_{loc, v, l}(x)$ depends continuously on $v, l$ in the $C^1$-topology for any $x \in A^*$. In particular, $W^s_{loc, v, l}(A)$  depends continuously on $v, l$ in the $C^1$-topology and  $ W^s_{loc, v_1^*, 0}(A) = W^s_{loc}(A)$.

    \begin{remark}
        Note that \cite[Theorem 4.1]{1970-HPS} is applicable to normally hyperbolic manifolds, while sets $S^*$, $A^*$ are topological discs. In order to be able to apply \cite[Theorem 4.1]{1970-HPS}, we suggest the following construction. Consider $U_{S^*}$ a neighborhood of $S^*$ (same as in definition of $W^u_{loc}(S^*)$) and vector fields $F_{v_1^*, 0}^*$, $\tilde{F}_{v,l}$ satisfying (we omit the construction)
        \begin{itemize}
            \item $F_{v_1^*, 0}^*(x) = F_{v_1^*, 0}(x)$, $\tilde{F}_{v,l}(x) = F_{v,l}(x)$ for $x \in U_{S^*}$,
            \item There are compact manifolds $N$, $N_{v,l}$ normally hyperbolic for $\tilde{F}_{v_1^*, 0}$ and $\tilde{F}_{v,l}$ respectively and $S^* \subset N$, $S^* \subset N_{v,l}$.
            \item $\tilde{F}_{v,l}$ is a small perturbation of $F_{v_1^*, 0}^*$.
        \end{itemize}
        Note that $W^u_{loc}(S^*)$ and $W^u_{loc, v, l}(S^*)$ would be unstable manifolds of $S^*$ for vector fields $F_{v_1^*, 0}^*$, $\tilde{F}_{v,l}$ as well, and \cite[Theorem 4.1]{1970-HPS} is applicable to the family of vector fields $\tilde{F}_{v,l}$.

        Similar construction is possible for the set $A^*$.

        See also \cite[Theorem 4.5.1, Theorem 5.6.1, Section 6.3]{Wiggins}, \cite{Sakamoto1990, JK1994, CKT1998, liu2000exchange}.
    \end{remark}

    Since $F_{v,l}$ depends continuously on $v,l$ in $C^1_{loc}$-topology, we see that the disk $\mathfrak{D}$ is transverse to the vector field $F_{v,l}$ for $l>0$ small enough and $v$ close to $v_1^*$. Similarly to~\eqref{eq:Wscomp} and~\eqref{eq:Wucomp}, we define:
    \begin{align*}
        W^s_{\mathrm{comp}, v, l}(A):&=\bigl\{\phi^{v,l}_{\xi} (x)\in\mathbb{R}^4\,:\, \text{for all }x\in W^s_{loc, v, l}(A), \,\xi\in[-\xi_0-1,0]\bigr\},
        \\
        W^u_{\mathrm{comp},v,l}(S^*) :&= \bigl\{\phi^{v,l}_{\xi}(x)\in\mathbb{R}^4\,:\,\text{for all } x\in W^u_{loc, v, l}(S^*), \,\xi\in[0,\xi_1+1]\bigr\}.  
    \end{align*}
Continuous dependence of $W^u_{loc, v, l}(S^*)$  on $v, l$  in the  $C^1$-topology implies continuous dependence of $W^u_{\mathrm{comp}, v, l}(S^*)$  on $v,l$ in the  $C^1$-topology. Similarly $W^s_{\mathrm{comp}, v, l}(A)$ depends continuously on $v,l$ in the  $C^1$-topology.

Transversality of intersection of the manifolds $W^s_{\mathrm{comp}}(A) = W^s_{\mathrm{comp}, v_1^*, 0}(A)$ and $W^u_{\mathrm{comp}}(S^*) = W^u_{\mathrm{comp}, v_1^*, 0}(S^*)$ implies that the set 
\begin{align*}
    \gamma_{v,l} := W^u_{\mathrm{comp}, v, l}(S^*) \cap W^s_{\mathrm{comp}, v, l}(A) \ne \emptyset,
\end{align*}
and depends continuously on $v,l$ in $C^1$-topology. Note that $\dim \gamma_{v,l} = 1$ and $\gamma_{v,l}$ is a trajectory of~\eqref{eq:4-dim-F-ve}. Since $\mathfrak{D}$ is transverse to the vector field $F_{v,l}$ and $\gamma_{v_1^*,0} \cap \mathfrak{D} = C$ then $C(v, l) := \gamma_{v,l} \cap \mathfrak{D}$ depends  continuously on $v, l$.
\end{proof}

\subsubsection*{Step 3: Continuous dependence of intersection point}
Let $D(v,l) := \phi_{-\xi_1-1}^{v,l}(C(v,l)) \in W^u_{loc, v, l}(S^*)$. Due to Lemma~\ref{lem:inte}, the point $D(v,l)$ depends continuously on $v$ and $l$. For small enough $l_0>0$ denote 
\begin{align*}
\mathds{W} :   = \{(v,l, x): \;  |v-v_1^*|< l_0,\;  l \in [0, l_0), \; x \in   W^u_{loc, v, l}(S^*) \}.
\end{align*}
Consider the map 
$$
P:\mathds{W} \to S^*, \quad\text{defined in such a way that } x \in W^u_{loc, v, l}(P(v,l, x)).
$$
Continuous dependence of unstable lamination \cite[Theorem 4.1(f)]{1970-HPS} implies that the map $P$ is uniquely defined and continuous. 
Hence, the point 
$$
B(v,l): = P(v,l, D(v,l))
$$
depends continuously on $v,l$ and 
there exists a heteroclinic trajectory $X(\xi)$ of~\eqref{eq:4-dim-F-ve} that connects $B(v,l) \in S^*$ and $A$. Hence, there exists a unique trajectory of~\eqref{eq:slow-system}, $(v, X(\xi), Y(\xi)) \in K_{l}$. Then $(X(\xi), Y(\xi))$ is a heteroclinic trajectory of \eqref{eq:6dim-dyn-sys-Peaceman-ver2} with given $(v, l)$. As the point $B(v,l) = (a_0(v, l), b_0(v, l), 0, 0)$ depends continuously on $v,l$, so do the functions $a_0(v,l)$, $b_0(v,l)$.

\subsubsection*{Step 4: Consistency}
First, we prove that the inequality $r_1\leq0$ holds on heteroclinic trajectory $\Gamma_{v, l}$, that connects $B(v,l)$ and $A$. Consider $\widehat\xi>0$ such that
$$
\phi_{-\widehat\xi} \,(C(v_1^*, 0)) \in W^u_{loc}(B), \quad \phi_{\widehat\xi} \,(C(v_1^*, 0)) \in W^s_{loc}(A).
$$
For $l\in[0,l_0)$ for small enough $l_0>0$ we can guarantee the inclusions
$$
\phi_{-\widehat\xi} \,(C(v, l)) \in W^u_{loc, v,l}(B(v, l)), 
\quad 
\phi_{\widehat\xi} \,(C(v, l)) \in W^s_{loc,v,l}(A).
$$
We prove $r_1\leq0$  at $\phi_{\xi} \,(C(v, l))$ independently for $\xi<-\widehat\xi$,\, $\xi \in [-\widehat\xi, \widehat\xi]$ and $\xi>\widehat\xi$.
\begin{enumerate}

    \item Consider $\xi< - \widehat\xi$. Note that the tangent vector at point $B$ to the manifold $W^u_{loc}(B)$ is (see \eqref{eq:13.1})
    \begin{equation}\label{eq:Add2.1}
        (-1, 1, -v_1^*, 2v_1^*)
    \end{equation}
    with non-zero $r_1$ component and heteroclinic trajectory $\Gamma_{v_1^*, 0}$ satisfies $r_1<0$, see \eqref{eq:parabola-1}. Hence, for small enough $l_0>0$ the tangent vector to $W^u_{loc}(B(v, l))$ at the point $B(v, l)$ is close to \eqref{eq:Add2.1}. Thus, the inequality $r_1< 0$ holds on $\Gamma_{v, l} \cap W^u_{loc}(B(v, l))$.
    
    \item Consider $\xi \in [-\widehat\xi, \widehat\xi]$. Note that interval $\xi \in [-\widehat\xi, \widehat\xi]$ is bounded and for the set $\{\phi_\xi(C(v_1^*, 0)), \; \xi \in [-\widehat\xi, \widehat\xi]\}$ the inequality $r_1< r^c$ holds for some constant $r^c<0$. Additionally decreasing $l_0$, we can conclude that the inequality $r_1< r^c/2 <0$ holds also for the set $\{\phi^{v, l}_\xi(C(v, l)), \; \xi \in [-\widehat\xi, \widehat\xi]\}.$
    
    \item Consider $\xi > \widehat\xi$. 
    Note that $T_A W^s_{loc}(A) \subset \{r_1 = -v_1^*a\}$, see \eqref{eq:EsA}. Hence for sufficiently small $l_0$ all points $(a, b, r_1, s_1) \in W^s_{loc, v, l}(A)$ satisfy inclusion $r_1 \in (-\frac{3}{2}v_1^*a, -\frac{1}{2}v_1^*a)$. So we can write $r_1(\xi) = k(\xi)a(\xi)$, where $k(\xi) \in (-\frac{3}{2}v_1^*, -\frac{1}{2}v_1^*)$, on the heteroclinic trajectory $\Gamma_{v,l}$. Thus,  $r_1(\xi)$ has a different sign than $a(\xi)$ and
    $$
    \dot{a} = \left(k(\xi) +u_1 +\frac{a}{2}\right)a, \quad \xi \geq \widehat\xi.
    $$
    This implies that $a(\xi)$ has the same sign for all $\xi \geq \widehat\xi$ and $r_1(\xi)$ has an opposite sign. Since $r_1(\widehat\xi) < 0$ then $r_1(\xi) \leq0$ for $\xi>\widehat\xi.$   
\end{enumerate}

We proved that the inequality $r_1\leq0$ holds on the heteroclinic trajectory $\Gamma_{v, l}$. Now we can prove that $q_1 \geq 0$ on  $\Gamma_{v, l}$ (by contradiction). Equations \eqref{eq:6dim-dyn-sys-Peaceman-ver2} imply 
\begin{equation} 
\label{eq:ddotq1}
q_1 = -\frac{1}{2}l r_1 -\frac{1}{4}l^2a\dot{q}_1+\frac{1}{2}l^2\ddot{q}_1.   
\end{equation}
Note that $q_1 \to 0$ as $\xi\to\pm \infty$. If $\inf_\xi q_1(\xi) < 0$, then there exists $\xi_2\in\mathbb{R}$ such that $q_1(\xi_2) = \inf_\xi q_1(\xi) < 0$. Equation~\eqref{eq:ddotq1} does hold for $\xi = \xi_2$. Since left hand side of~\eqref{eq:ddotq1} is negative at $\xi = \xi_2$, and  the right hand side of~\eqref{eq:ddotq1} is non-negative
$$
-\frac{1}{2}l r_1(\xi_2) \geq 0, \quad \dot{q}_1(\xi_2) = 0, \quad \ddot{q}_1(\xi_2) \geq0,
$$
we get a contradiction. Hence, $\inf_\xi q_1(\xi) \geq 0$, and Item 1b of Theorem~\ref{thm:Peaceman-reformulated} is proved.
\end{proof}

\subsection{Existence of propagating terrace}
\label{subsec:prop-terrace-IPM}
In this section we prove Theorem~\ref{thm:main-Peaceman}. By Theorem~\ref{thm:TW-IPM} we obtain that there exists sufficiently small $l_0,\delta >0$ such that:
\begin{itemize}
    \item  for all $v\in(v_1^*-\delta,v_1^*+\delta)$ and  $l\in(0,l_0)$ there exists a traveling wave solution $(c_{1},c_2,u_1,u_2,q)(y-vt)$ that connects the states 
\begin{align*}
(-1,-1,0,0,0) \xrightarrow{TW} \left(c_1^*,c_2^*,\frac{c_2^*-c_1^*}2,\frac{c_1^*-c_2^*}2,0\right).
\end{align*}
    
Recall that $\mathcal{H}_{-1}(l):=\{(c_1^*(v,l),c_2^*(v,l)): v\in(v_1^*-\delta,v_1^*+\delta)\}$, see Fig.~\ref{fig:HugoniotIPM}. 
 \item  
 for all $v\in(v_2^*-\delta,v_2^*+\delta)$ and $l\in(0,l_0)$ there exists a traveling wave solution 
$(c_{1},c_2,u_1,u_2,q)(y-vt)$ that connects the states
\begin{align*}
 \left(c_1^{**},c_2^{**},\frac{c_2^{**}-c_1^{**}}2,\frac{c_1^{**}-c_2^{**}}2,0\right)\xrightarrow{TW} (1,1,0,0,0),
\end{align*}
Recall that $\mathcal{H}_{1}(l):=\{(c_1^{**}(v,l),c_2^{**}(v,l)): v\in(v_2^*-\delta,v_2^*+\delta)\}$, see Fig.~\ref{fig:HugoniotIPM}.
\end{itemize}

Note that $\mathcal{H}_{-1}(0):
    =\bigl\{\lim\limits_{l\to0}(c_1^*(v,l),c_2^*(v,l)): v\in(v_1^*-\delta,v_1^*+\delta)\bigr\}$
coincides with the subset of $\mathcal{H}_{-1}$ under the restriction $v\in (v_1^*-\delta,v_1^*+\delta)$; also the set $\mathcal{H}_{1}(0):=\bigl\{\lim\limits_{l\to0}(c_1^{**}(v,l),c_2^{**}(v,l)): v\in(v_2^*-\delta,v_2^*+\delta)\bigr\}$ coincides with the subset of $\mathcal{H}_{1}$ under the restriction $v\in (v_2^*-\delta,v_2^*+\delta)$. From Section~\ref{subsec:heteroclinic-orbits} we know that the point $(\frac12,-\frac12)$ lies in the intersection of $\mathcal{H}_1(0)$ and $\mathcal{H}_{-1}(0)$. For every sufficiently small  $l\in[0,l_0)$ the set $\mathcal{H}_1(l)$ (similarly, $\mathcal{H}_{-1}(l)$) is a curve and is close to $\mathcal{H}(0)$ and $\mathcal{H}_{-1}(0)$ near their intersection point.

Let us give an easy topological argument that 
shows that the curves $\mathcal{H}_{1}(l)$ and $\mathcal{H}_{-1}(l)$ have at least one point of intersection. Indeed, consider a rectangle $R:=\{(c_1,c_2): |c_1-1/2|<\delta_1, |c_2+1/2|<\delta_2\}$ for small enough $\delta_{1,2}>0$ such that the curve $\mathcal{H}_1(0)$ intersects $R$ in the interior of the top and bottom sides, and the curve $\mathcal{H}_{-1}(0)$ intersects $R$ in the interior of the left and right sides, see Fig.~\ref{fig:HugoniotIPM}. Taking $l>0$ small enough we have that the curve $\mathcal{H}_1(l)$ is close to $\mathcal{H}_1(0)$ and also intersects $R$ in top and bottom sides. Similarly, the curve $\mathcal{H}_{-1}(l)$ is close to $\mathcal{H}_{-1}(0)$ and intersects $R$ in left and right sides. Thus, $\mathcal{H}_1(l)$ and $\mathcal{H}_{-1}(l)$ intersect.

Hence there exist $v_1^*(l)$ and $v_2^*(l)$ such that the intersection $\mathcal{H}_1(l)\cap \mathcal{H}_{-1}(l)$ contains the point $(c_1^*(v_2^*(l),l),c_2^*(v_2^*(l),l))$ (or the same point $(c_1^{**}(v_1^*(l),l),c_2^{**}(v_1^*(l),l))$).  By continuity argument, we obtain
\begin{align*}
        &\lim\limits_{l\to 0} c_1^*(v_2^*(l),l)=-1/2,\quad \lim\limits_{l\to 0} c_2^*(v_2^*(l),l)=1/2, \quad \lim\limits_{l\to 0} v_1^*(l)=-\lim\limits_{l\to 0} v_2^*(l)=1/4,
\end{align*}
 thus Theorem~\ref{thm:main-Peaceman} is proven.

\section{Discussions and generalizations}
\label{sec:discussions}
There are several directions for further investigation:

\mbox{\textbf{1.}\;} \textbf{Stability properties of the propagating terraces. }Note that Theorem~\ref{thm:main-Peaceman} establishes the existence of the propagating terrace consisting of two traveling waves for small values of parameter~$l$. Nevertheless, numerical simulations show that for a typical initial data, the solution of \eqref{eq:2tubes-diffusive-system-1}--\eqref{eq:IPM-2-tubes-w} converges to a propagating terrace for any value of $l >0$. This observation suggests that a statement of Theorem~\ref{thm:main-Peaceman} holds for any $l>0$, moreover, this propagating terrace is stable (in some suitable sense that needs to be defined more precisely). At the same time, for system \eqref{eq:2tubes-diffusive-system-1}--\eqref{eq:IPM-2-tubes-w} it is possible to construct other propagating terraces consisting of $k$ traveling waves for any $k> 2$, but we expect them to be unstable to small perturbations. For some references on stability of propagating terraces see Remark \ref{rem:pt-names}.

\mbox{\textbf{2.} \;} \textbf{$n$-tubes model for gravitational fingering. }It is possible to construct a model on similar principles for $n>2$ tubes. In that case for typical initial data, the solution converges to a propagating terrace consisting of several traveling waves. However, for $n>2$ limiting solution profile is not unique and often consists of $k \geq n$ traveling waves. 
For short description of the construction and its properties see~\cite{Oberwolfach-2024}.

\mbox{\textbf{3.} \;}\textbf{Two- and $n$-tubes model for viscous fingering. }
There is a natural analogue of a two-tubes model for the system \eqref{eq:IPM-1}, \eqref{eq:IPM-2}, \eqref{eq:Darcy-viscosity} that describes viscous fingering phenomenon. Numerical simulations show that for typical initial data solution of a two-tubes model converges to a propagating terrace of two traveling waves. This observation suggests that analogue of Theorem \ref{thm:main-Peaceman} for the case of viscous fingers should be correct. Most of the steps of the proof of Theorem \ref{thm:main-Peaceman} could be repeated in that case. However, we cannot prove the existence of analogues of invariant hyperplanes $I_1,  I_4$. At the same time numerics shows that corresponding dynamical system has invariant sets which play a similar role. Detailed analysis of 2- and $n$-tube model for viscous liquids displacement is a subject for future research.

\section*{Acknowledgement}
We thank Aleksandr Enin for fruitful discussions. Research of ST and YP is supported by Projeto Paz and Coordenacao de Aperfeicoamento de Pessoal de Nivel Superior - Brasil (CAPES) - 
23038.015548/2016-06. ST is additionally supported by FAPERJ PDS 2021, process code E-26/202.311/2021 (261921),  CNPq through the grant  404123/2023-6  and FAPERJ APQ1 E-26/210.702/2024 (295291). 
YP was additionally supported by CNPq through the grant 406460/2023-0 and FAPERJ APQ1 E-26/210.700/2024 (295287). Part of the work was written while YP was a postdoctoral fellow at IMPA. ST and YP thank IMPA and PUC-Rio for creating excellent working conditions.  

\appendix
\section{Derivation of the two-tubes IPM model}
\label{ap:FV}
In this Section we give a schematic derivation of the two-tubes IPM model~\eqref{eq:2tubes-diffusive-system-1}--\eqref{eq:IPM-2-tubes-w} using the finite-volume scheme (see~\cite{Finite-volume-2000}) under the assumption that all the unknown functions are sufficiently smooth. Our plan is as follows:
\begin{enumerate}
    \item[(1)] Write down a finite volume scheme for the 2D inviscid IPM model~\eqref{eq:IPM-1}--\eqref{eq:IPM-3} in a strip $[0,2l]\times\mathbb{R}$ on rectangular mesh with two cells in the horizontal direction (assuming $\nu=0$). 

    \item[(2)] Consider the cells of sizes $l\times h$, and take a formal limit as $h\to0$, which results in two 1D equations~\eqref{eq:2tubes-diffusive-system-1}--\eqref{eq:2tubes-diffusive-system-2} for concentrations coupled with each other due to equations on velocity and pressure~\eqref{eq:peretok}--\eqref{eq:IPM-2-tubes-w}. 

    \item[(3)] Add diffusion terms $c_{yy}$ to equations~\eqref{eq:2tubes-diffusive-system-1}--\eqref{eq:2tubes-diffusive-system-2}. So formally the two-tubes IPM model contains diffusion only in vertical direction (direction of gravity).
\end{enumerate}

\begin{minipage}{0.52\textwidth}
 Let us consider discretization grid $2 \times \mathbb{Z}$ with cells of sizes $l\times h$
\begin{align*}
    T_{i,j}=[(i-1)l,il]\times[(j-1/2)h,(j+1/2)h].
\end{align*}
Here $(i,j)\in\{1,2\}\times\mathbb{Z}$.
Let $C_{i,j}^n$, $P_{i,j}^n$, $i\in\{1,2\}$, $j\in\mathbb{Z}$, $n\in\mathbb{N}$, be the discrete unknowns associated to the cell $T_{i,j}$ at time step $t_n\geq0$. 

Initially, at $t_0=0$, we define:
\begin{align*}
    C_{i,j}^0=\frac{1}{|T_{i,j}|}\int\limits_{T_{i,j}} c(0,x,y)\, dxdy.
\end{align*}

\end{minipage}
\hfill
\begin{minipage}{0.37\textwidth}
\includegraphics[width=0.9\textwidth]{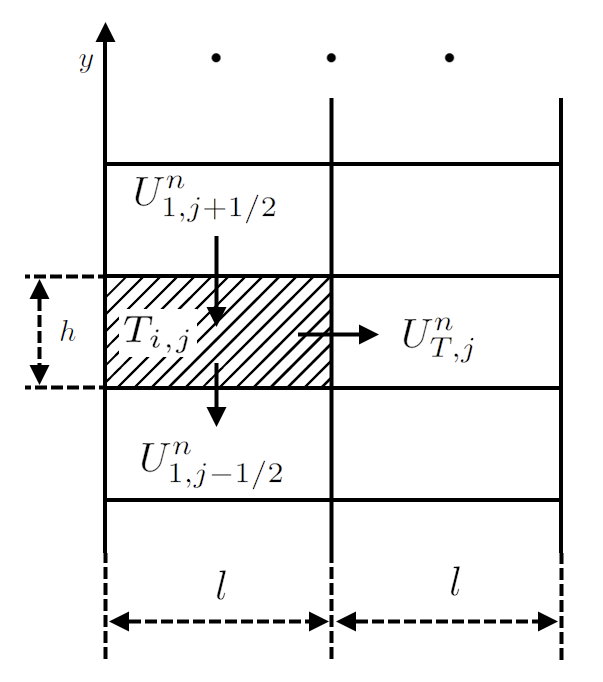}
 \captionof{figure}{Finite-volume scheme with two columns of cells.}
    \label{fig:FV-scheme}
\end{minipage}

At the boundaries $x=0$ and $x=2l$ we assume no-flux boundary condition, and define the fluxes through the other edges of each cell using an upwind scheme.

 Integrating the equation~\eqref{eq:IPM-1} over the cell, we obtain:
\begin{align*}
    \partial_t \Bigl(\int\limits_{T_{i,j}} c\,dxdy \Bigr)- \int\limits_{\partial T_{i,j}} u c \cdot n \,dl=0.
\end{align*}
Following the finite volume scheme~\cite{Finite-volume-2000}, we obtain the discrete equation  ($\Delta t=t_{n+1}-t_n$, $n\in\mathbb{N}$)
\begin{align}
\label{eq:FV-transport}
    \frac{C_{i,j}^{n+1}-C_{i,j}^{n}}{\Delta t} &+\frac{U_{i,j-1/2}^nC_{i,k_{i, j-1/2}^n}^n - U_{i,j+1/2}^nC_{i,k_{i, j+1/2}^n}^n}{h}+ \frac{ U_{T,j}^n C_{m_j^n,j}^n}{l}=0,
\end{align}
where $U_{i,j-1/2}^n$ and $U_{i,j+1/2}^n$, $(i,j)\in\{1,2\}\times\mathbb{Z}$, are the normal velocities defined on the upper and lower edges of the cell $T_{i,j}$ at time $t_n$; $U_{T,j}^n$ is the normal velocity defined on the edge between the cells $T_{1,j}$ and $T_{2,j}$  at time $t_n$. The velocities $U_{i,j+1/2}^n$ and $U_{T,j}^n$ are defined below in~\eqref{eq:FV-Div}--\eqref{eq:FV-Darcy}. We use the following notation encapsulating upwind scheme:
\begin{align*}
    k_{i, j+1/2}^n & = \begin{cases}
        j, & \mbox{if $U_{i,j+1/2}^n \geq 0$}
        \\
        j+1,& \mbox{if $U_{i,j+1/2}^n \leq 0$};
    \end{cases}
    &&
    m_{j}^n  =
    \begin{cases}
        1, &\mbox{if $U_{T,j}^n \geq 0$},
        \\
        2, &\mbox{if $U_{T,j}^n \leq 0$}.
    \end{cases}
\end{align*}

The incompressibility condition~\eqref{eq:IPM-2} can be written as:
\begin{align}\label{eq:FV-Div}
    \frac{U_{1,j+1/2}^n-U_{1,j-1/2}^n}{h}+\frac{U_{T,j}^n}{l}=0, \qquad \frac{U_{2,j+1/2}^n-U_{2,j-1/2}^n}{h}-\frac{U_{T,j}^n}{l}=0.
\end{align}

The Darcy's law~\eqref{eq:IPM-3} can be written as:
\begin{align*}
  \int\limits_{\Gamma} (u+\nabla p+(0,c))\cdot \vec{n}\, dl=0,
\end{align*}
where $\Gamma$ is one of the edges of the cell $T_{i,j}$ with $\vec{n}$ being its normal vector.   
In a discrete setting, this is equivalent to the following system of equations written for each direction:
\begin{align}\label{eq:FV-Darcy}
    U_{i,j+1/2}^n=-\frac{P_{i,j+1}^n-P_{i,j}^n}{h}-C_{i, j+1/2}^n, \qquad U_{T,j}^n=-\frac{P_{2,j}^n-P_{1,j}^n}{l},
\end{align}
where $C_{i, j+1/2}^n$ is the value of concentration on the edge between $T_{i, j}$, $T_{i, j+1}$. Since for finite volume scheme value $C_{i, j}$ are defined for $j \in \mathbb{Z}$, we take 
$
C_{i, j+1/2}^n = \frac{C_{i,j}^n+C_{i,j+1}^n}2.
$

\begin{remark}
    In finite volume methods, the normal velocity is approximated  by pressures at the two points which are located on the orthogonal bisector on each side of the edge (see~\cite{Finite-volume-2000}), which can be challenging for non-orthogonal grids or tensor permeabilities. For our orthogonal grid,  normal velocities are  simply directional components of the velocity field and thus, we only need to use directional derivatives of the pressure field.
\end{remark}

Consider sufficiently smooth functions $c_{1,2}(t,y)$, $p_{1,2}(t,y)$ and $u_{1,2,T}(t,y)$ with $(t,y)\in\mathbb{R}_+\times\mathbb{R}$,  satisfying 
\begin{align*}
    c_{i}(t_n,jh)=C_{i,j}^n;\qquad p_{i}(t_n,jh)=P_{i,j}^n, \quad j\in\mathbb{Z},\,n\in\mathbb{N}, \,i\in\{1,2\},
    \\
    u_{i}(t_n,(j+1/2)h)=U_{i,j+1/2}^n, \quad j\in\mathbb{Z},\,n\in\mathbb{N}, \,i\in\{1,2,T\},
\end{align*}
and take in equations \eqref{eq:FV-transport}--\eqref{eq:FV-Darcy} a formal limit  as $\Delta t\to 0$ and $h \to 0$. From equations \eqref{eq:FV-transport} we obtain the equations~\eqref{eq:2tubes-diffusive-system-1}--\eqref{eq:peretok} with zero diffusion term; from equation \eqref{eq:FV-Div} we obtain semi-discrete incompressibility condition~\eqref{eq:IPM-2-tubes-w}; and from  equation \eqref{eq:FV-Darcy} we obtain the semi-discrete Darcy's law~\eqref{eq:gravitational-Darcy-2tubes-a}--\eqref{eq:gravitational-Darcy-2tubes-b}.

\section{Connection between TFE and IPM models} 
\label{ap:TFE-IPM}
In this Section we describe the connection between the Incompressible Porous Medium (IPM) equations~\eqref{eq:IPM-1}--\eqref{eq:IPM-3} and the Transverse Flow Equilibrium (TFE) equations~\eqref{eq:IPM-1}--\eqref{eq:IPM-2},~\eqref{eq:TFE-2D-u}. We want to show that TFE model can be seen as a first-order approximation to the IPM model in the regime when the flow ``equilibrates much faster'' in the transversal direction than in the main direction of the flow (e.g. an asymptotically thin domain). Notice that both models exhibit a fingering instability, but up to our knowledge, there is no rigorously proven connection between TFE and IPM. Nevertheless, for the viscous fingering the connection was formally derived in~\cite{yortsos1995-TFE}; see also~\cite[Section 3.3]{armiti-rohde-2019} for numerical comparison of TFE and IPM.

Consider the domain $(x,y)\in(0,L)\times(0,H)$, and a system~\eqref{eq:IPM-1}--\eqref{eq:IPM-2} under the assumption $\nu=0$, supplied with a more general form of the Darcy's law (here $k^x, k^y>0$ are permeabilities in $x$ and $y$-directions):
\begin{align}
\label{eq:IPM-3c}
        u=(u^x, u^y) =-
    \left(\begin{matrix}
k^x & 0 \\
0 & k^y 
\end{matrix}\right)\nabla p -(0,c).
\end{align}
First, we make the following change of variables:
\begin{align*}
    \tilde{x}=\frac{x}{L}\in[0,1], 
    \qquad 
    \tilde{y}=\frac{y}{H}\in[0,1], 
    \qquad 
    \tilde{t}=\frac{t}{H},
    \qquad 
    \tilde{p}=\frac{k^y}{H}\cdot p,
    \qquad 
    \tilde{u}^x=\frac{H}{L} \cdot u^x,
\end{align*}
 and the IPM equations~\eqref{eq:IPM-1}--\eqref{eq:IPM-2},~\eqref{eq:IPM-3c} become (here $\gamma:=\frac{L}{H}\sqrt{\frac{k^y}{k^x}}$; we omit tildas):
\begin{align}
\label{eq:IPM-gamma-1}
    \partial_t c+\partial_x(u^xc)+\partial_y(u^yc)&=0,\quad
    \\
\label{eq:IPM-gamma-2}
    \partial_x u^x+\partial_y u^y&=0,
    \\
\label{eq:IPM-gamma-3}
    u^x&= -\frac{1}{\gamma^2}\partial_x p,
    \\
\label{eq:IPM-gamma-4}
    \quad u^y&= -\partial_y p-c.
\end{align}
This form of equations suggests to write a formal asymptotic expansion in $\gamma^2$ as $\gamma\to0$:
\begin{align*}      g(t,x,y)&=g^0(t,x,y)+\gamma^2\cdot g^1(t,x,y)+\ldots
\end{align*}
for $g=\{c,p,u^x,u^y\}$. 
 Using~\eqref{eq:IPM-gamma-3}, we obtain $\partial_x p^0=0$, thus up to first approximation $p$ doesn't depend on $x$ variable. Integrating~\eqref{eq:IPM-gamma-4} and using no-flow conditions on all boundaries, we get
\begin{align*}
    0=\int\limits_{0}^1u^{y,0}(t,x,y)\,dx=-\partial_yp^{0}(t,y) -\int\limits_0^1c^0(t,x,y)\,dx.
\end{align*}
Putting this relation into~\eqref{eq:IPM-gamma-4},  we immediately get~\eqref{eq:TFE-2D-u} for velocity $u^{0}=(u^{x,0}, u^{y,0})$. Note that first-order approximation of the equations~\eqref{eq:IPM-gamma-1}--\eqref{eq:IPM-gamma-2} give exactly the equations~\eqref{eq:IPM-1}--\eqref{eq:IPM-2} for $c^0$ and $u^0$ (under the assumption $\nu=0$). This finishes the derivation of the TFE equations~\eqref{eq:IPM-1}--\eqref{eq:IPM-2},~\eqref{eq:TFE-2D-u}.

    The parameter $\gamma=\frac{L}{H}\sqrt{\frac{k^y}{k^x}}$, responsible for the connection between IPM and TFE models, was originally introduced by Yortsos in~\cite{yortsos1995-TFE}. It allows for the following two interpretations:
    \begin{itemize}
        \item If $k^x=k^y$ and $\gamma=\frac{L}{H}\to0$, then         the transverse length of the domain is much smaller compared to the length in the main direction of the flow (e.g.~\cite{armiti-rohde-2014,armiti-rohde-2019}).
        \item If $L=H$ and $\gamma=\sqrt{\frac{k^y}{k^x}}\to0$, then the permeability in the transverse direction is much higher than the permeability in the main direction of the flow.
    \end{itemize} 
    In our context of the two-tubes models, the parameter $l$ (distance between the tubes) plays the same role as parameter $\gamma$, thus the limit $l\to0$ corresponds to the two-tubes TFE model, as expected.

\section{Slowdown of fingers due to intermediate concentration} 
\label{ap:slowdown}
 In this Section we clarify what we mean by ``slowdown of fingers due to intermediate concentration''. We give two examples: for 2D and two-tubes TFE~models.
    
\subsection{2D case}
\label{subsec:TFE-2D-slowdown}
In~\cite{otto-menon-2005} the authors prove pointwise upper and lower estimates on the speed of the growth of the mixing zone for 2D TFE model~\eqref{eq:IPM-1}--\eqref{eq:IPM-2} and \eqref{eq:TFE-2D-u}. 
    The key idea lies in comparison with a traveling wave solution connecting $-1$ and~$1$ of the following 1D viscous conservation law of Burger's type (for initial data~\eqref{eq-cpm1}):
\begin{align}
\label{eq:comp-Burgers}
    c_t+\left(c-\frac{c^2}{2}\right)_y=c_{yy}.
\end{align}
    By Rankine-Hugoniot condition, the speed of the traveling wave connecting states $-1$ and $1$, is equal to $v^f=1$, which corresponds to the mixing zone growth rate $h(t)\leq2t$ as was mentioned in the Introduction, see estimate~\eqref{eqat}.

    If on the tip of the finger the concentration drops down from 1 to some intermediate concentration $c^*\in(-1,1)$ (see numerical evidence~\cite{2023-intermediate-Tikhomirov} for viscous fingering), then the estimate~\eqref{eqat} can be improved. Indeed, adhering to the proof strategy of~\cite[Theorem~1]{2023-intermediate-Tikhomirov}, we compare the solution with the traveling wave connecting states $c^*$ and $1$ and get the following expression for its speed $v^f(c^*)$ due to Rankine-Hugoniot condition for the (inviscid version of) equation~\eqref{eq:comp-Burgers}:
\begin{align*}
    v^f(c^*)=\frac{\left[c-\frac{c^2}{2}\right]\biggr\rvert_{c^*}^{1}}{c\bigr\rvert_{c^*}^{1}}=\frac{1-c^*}{2}<1, \quad c^*\in(-1,1).
\end{align*}
In particular, if $c^*=-1$ we get $v^f(-1)=1$, which is in accordance with the estimate~\eqref{eqat}. 
So the appearance of the intermediate concentration on the tip of the finger reduces its speed of propagation and gives a slowdown  
(for 2D TFE model).
\subsection{Two-tubes case} 
\label{subsec:TFE-2tubes-slowdown}
One of the most natural guesses for a possible  propagating terrace consisting of two traveling waves for the two-tubes TFE model~\eqref{eq:2tubes-diffusive-system-1}--\eqref{eq:peretok},~\eqref{eq:IPM-2-tubes-w},~\eqref{eq:gravitational-Darcy-TFE-2tubes} is as follows:
\begin{align*}
    (-1,-1) \xrightarrow{TW} (-1,1)\xrightarrow{TW} (1,1).
\end{align*}
It can be interpreted as two counter-propagating fingers with concentrations $1$ and~$-1$.
If such propagating terrace exists, the speed of the  traveling wave connecting $(-1,1)$ and $(1,1)$ can be calculated using the following formula (here we use $u_1=-u_2=(c_2-c_1)/2$, see~\eqref{eq:gravitational-Darcy-TFE-2tubes})
\begin{align}
    \label{eq:RH-TFE}
v^f=\frac{\left[u_1c_1+u_2c_2\right]\bigr\rvert_{(c_1,c_2)=(-1,+1)}^{(c_1,c_2)=(+1,+1)}}
{\left[c_1+c_2\right]\bigr\rvert_{(c_1,c_2)=(-1,+1)}^{(c_1,c_2)=(+1,+1)}}=1.
\end{align}
This expression is an analogue of the Rankine-Hugoniot condition for the sum of  equations~\eqref{eq:2tubes-diffusive-system-1} and~\eqref{eq:2tubes-diffusive-system-2}. Indeed, this sum can be written in a conservative form:
\begin{align*}
\partial_t(c_1+c_2)+\partial_y(c_1u_1+c_2u_2)=\partial_{yy}(c_1+c_2).
\end{align*}
Putting traveling wave ansatz $(c_1,c_2)(t,y)=(c_1,c_2)(\xi)$, $\xi=y-vt$, integrating from $\xi=-\infty$ to $\xi=+\infty$ and assuming $\partial_y(c_{1,2})(\pm\infty)=0$, we obtain~\eqref{eq:RH-TFE} for~$v=v^f$.

Comparing the speed $v^f=1$ from formula~\eqref{eq:RH-TFE} with the speed $v^f=1/4$ obtained in Theorem~\ref{thm:main-TFE}, we observe that the selected propagating terrace for the two-tubes TFE model, indeed, moves slower. We interpret this propagating terrace as two counter-propagating fingers with concentrations $1/2$ and $-1/2$, and thus conclude that we have a slowdown of fingers growth due to intermediate concentration.


\end{document}